\documentclass[a4paper,fleqn]{cas-sc}

\usepackage[T1]{fontenc}
\usepackage[utf8]{inputenc}
\usepackage{amsmath,amssymb,amsthm,mathtools}
\usepackage{enumitem}
\usepackage{booktabs}
\usepackage{array}
\usepackage{xurl}
\usepackage[numbers,sort&compress]{natbib}

\setlist[enumerate]{itemsep=2pt, topsep=4pt}
\setlist[itemize]{itemsep=2pt, topsep=4pt}
\allowdisplaybreaks

\newtheorem{theorem}{Theorem}[section]
\newtheorem{proposition}[theorem]{Proposition}
\newtheorem{lemma}[theorem]{Lemma}
\newtheorem{corollary}[theorem]{Corollary}
\theoremstyle{definition}
\newtheorem{definition}[theorem]{Definition}
\theoremstyle{remark}
\newtheorem{remark}[theorem]{Remark}

\newcommand{\rhoext}{\rho'}
\newcommand{\Ap}{A^{+}}
\newcommand{\Az}{A^{0}}
\newcommand{\Smk}{S^{-}_{\frac{m+4}{2},2}}

\newcommand{\Lm}{L_m}

\begin{document}
\let\WriteBookmarks\relax
\def\floatpagepagefraction{1}
\def\textpagefraction{.001}

\shorttitle{Sharp even-size threshold for $H(4,3)$-free graphs}
\shortauthors{S. Sarkar}

\title[mode=title]{The Sharp Even-Size Spectral Threshold for $H(4,3)$-Free Graphs}

\author[1]{Shreyhaan Sarkar}[orcid=0009-0002-4325-9939]
\cormark[1]
\ead{sms736@cornell.edu}
\credit{Conceptualization, methodology, formal analysis, software, writing -- original draft}

\affiliation[1]{organization={Department of Mathematics, Cornell University},
                addressline={310 Malott Hall},
                city={Ithaca},
                postcode={14853},
                state={New York},
                country={USA}}

\cortext[cor1]{Corresponding author}

\begin{abstract}
Let $H(4,3)$ be the graph obtained by identifying one vertex of a $4$-cycle with one vertex of a triangle. We determine the sharp even-size threshold for the fixed-size spectral extremal problem forbidding $H(4,3)$. Specifically, if $G$ is an $H(4,3)$-free graph of even size $m\ge 18$ without isolated vertices, then $\rho(G)\le \rho'(m)$, where $\rho'(m)$ is the largest root of $x^4-mx^2-(m-2)x+m/2-1=0$. Equality holds if and only if $G\cong S^{-}_{(m+4)/2,2}$. The value $18$ is best possible: explicit $H(4,3)$-free obstruction graphs exceed the comparison value at $m=10,12,14,16$.

The main new ingredient is a local interface independence principle in the Perron-neighborhood decomposition. In the threshold-bearing branch where the neighborhood core induces $K_4$, the exterior vertices adjacent to this core form an independent set. This turns the only dangerous interface term into a finite endpoint problem and yields the best possible threshold.
\end{abstract}

\begin{keywords}
spectral radius \sep forbidden subgraph \sep extremal graph \sep Perron vector \sep graph spectra
\MSC[2020]{05C50 \sep 05C12 \sep 15A18}
\end{keywords}

\maketitle

\section{Introduction}

Let $G$ be a simple graph with $m$ edges, and let $\rho(G)$ denote its adjacency spectral radius. A fixed-size spectral extremal problem asks for the maximum possible value of $\rho(G)$ among graphs of size $m$ satisfying a prescribed forbidden-subgraph condition. This is the Brualdi--Hoffman--Tur\'an viewpoint, initiated by Brualdi and Hoffman~\cite{BrualdiHoffman1985} and closely related to classical spectral Tur\'an theory. The fixed-size spectral viewpoint also appears in Nosal's triangle-free bound~\cite{Nosal1970} and in Nikiforov's work on graphs without specified paths and cycles~\cite{Nikiforov2010}. For broader recent context, see Liu and Wang~\cite{LiuWang2024}, Li, Zhao, and Zou~\cite{LiZhaoZou2025}, and Rehman and Pirzada~\cite{RehmanPirzada2025}.

We study the graph $H(4,3)$ obtained by identifying one vertex of a $4$-cycle with one vertex of a triangle. This graph is often called the \emph{fish graph}. The fixed-size $H(4,3)$-free problem is a natural even-cycle analogue of the bowtie problem for $H(3,3)$. In the even case, the expected extremal graph is not the usual join $S_{(m+3)/2,2}$, but the edge-deleted graph $S^{-}_{(m+4)/2,2}$, whose spectral radius is the largest root of
\[
p_m(x)=x^4-mx^2-(m-2)x+\frac m2-1.
\]

The surrounding results point to the same extremal graph and polynomial. Table~\ref{tab:prior-results} summarizes the thresholds most relevant to the present paper.
\begin{center}
\small
\refstepcounter{table}\label{tab:prior-results}
\textbf{Table~\thetable}\par
\medskip
\begin{minipage}{0.96\textwidth}
\centering
Comparison with prior fixed-size spectral extremal results relevant to the even-size $H(4,3)$-free problem.\par
\medskip
\begin{tabular}{@{}p{0.31\textwidth} c c p{0.20\textwidth} p{0.15\textwidth}@{}}
\toprule
Forbidden condition & Parity & Threshold & Extremal graph & Source \\
\midrule
$H(3,3)$-free & odd & $m\ge 8$ & $S_{\frac{m+3}{2},2}$ & Li--Lu--Peng~\cite{LiLuPeng2023} \\
$\{H(3,3),H(4,3)\}$-free & even & $m\ge 10$ & $S^{-}_{\frac{m+4}{2},2}$ & Pirzada--Rehman~\cite{PirzadaRehman2025} \\
$H(4,3)$-free & even & $m\ge 51$ & $S^{-}_{\frac{m+4}{2},2}$ & Zhang--Wang~\cite{ZhangWangFish2025} \\
$H(4,3)$-free & even & $m\ge 38$ & $S^{-}_{\frac{m+4}{2},2}$ & Zheng--Zhang~\cite{ZhengZhang2025} \\
$H(4,3)$-free & even & $m\ge 18$, sharp & $S^{-}_{\frac{m+4}{2},2}$ & this paper \\
\bottomrule
\end{tabular}
\end{minipage}
\end{center}
Thus the extremal graph and polynomial were already strongly suggested by the existing theory. The remaining question was the exact even-size threshold for forbidding $H(4,3)$ alone.

Our main theorem resolves that threshold.

\begin{theorem}\label{thm:main}
Let $G$ be an $H(4,3)$-free graph of even size $m \ge 18$ without isolated vertices. Then
\[
\rho(G)\le \rhoext(m),
\]
where $\rhoext(m)$ is the largest real root of
\[
p_m(x)=x^4-mx^2-(m-2)x+\frac{m}{2}-1.
\]
Moreover, equality holds if and only if $G \cong \Smk$.
\end{theorem}

The value $18$ is sharp. For $m=10,12,14,16$, the graphs
\[
T_m=K_1\vee\bigl(K_4\cup(m-10)K_1\bigr)
\]
are $H(4,3)$-free and satisfy $\rho(T_m)>\rho'(m)$. Thus the theorem cannot be extended to any smaller even threshold. Proposition~\ref{prop:ew0} proves this obstruction and also shows that $\rho(T_m)<\rho'(m)$ for every even $m\ge18$.

The main new ingredient is a local independence principle for the interface between the Perron-neighborhood core and the exterior set. In the standard decomposition around a maximum Perron-vector vertex $u^*$, write
\[
N=N(u^*),\qquad W=V(G)\setminus N[u^*],
\]
and let
\[
A^+=\{v\in N:d_N(v)\ge1\}.
\]
The only branch that can carry the threshold is the branch $G[A^+]\cong K_4$. In this branch we prove that
\[
U=\{w\in W:d_{A^+}(w)=1\}
\]
is an independent set in $G[W]$. Therefore
\[
e(A^+,W)=|U|\le \alpha(G[W]).
\]
This converts the interface term $e(A^+,W)$, which is the obstruction in the Perron inequalities, into a purely local invariant of the small graph $G[W]$. Since the Perron inequalities already force $e(W)\le3$ in the $K_4$-branch, the independence principle makes the remaining analysis finite and sharply localized. The only genuine endpoint work occurs when $e(W)=2$, or when $e(W)=3$ and $m=18$.

Our proof uses the Perron-neighborhood decomposition developed by Zheng and Zhang~\cite{ZhengZhang2025} as a starting point. Some parts of their branch reduction are threshold-independent, while others use inequalities valid only in their stated range $m\ge38$. For the sharp threshold $18$, we reprove the reductions needed below that range, including the non-$K_4$ endpoint branches and the terminal star branch. Thus the dependence on Zheng--Zhang is organizational rather than threshold-critical.

\subsection*{Organization}
Section~\ref{sec:prelim} fixes notation, identifies the comparison graph, and proves the Perron-neighborhood inequalities used later. Section~\ref{sec:reduction} reduces the proof to the unique threshold-bearing $K_4$-branch and includes a proof map for the reduction. Section~\ref{sec:local} proves the interface independence principle. Sections~\ref{sec:ew0}--\ref{sec:ew3} treat the cases $e(W)=0,1,2,3$ in the $K_4$-branch. Section~\ref{sec:proof-main} proves the main theorem. Appendix~\ref{app:signchecks} contains the exact polynomial checks used in the finite quotient comparisons and explains the accompanying verification script.

\section{The Perron-neighborhood setup}\label{sec:prelim}

All graphs are finite, simple, and undirected. For a graph $G$, we write $V(G)$ and $E(G)$ for its vertex and edge sets and $\rho(G)$ for its adjacency spectral radius.

\begin{lemma}[Connectedness reduction]\label{lem:connected-reduction}
Among all $H(4,3)$-free graphs of even size $m\ge18$ without isolated vertices that maximize the spectral radius, there is a connected maximizer.
\end{lemma}

\begin{proof}
Let $G$ be an extremal graph, and suppose that $G$ is disconnected. Let $C$ be a component of $G$ with
\[
\rho(C)=\rho(G),
\]
and put $r=m-e(C)$. Since $G$ is disconnected, $r>0$. Choose a vertex $u\in V(C)$, and form a new graph $G'$ by taking $C$ and adding $r$ new leaves, each adjacent only to $u$. Then $G'$ is connected, has exactly $e(C)+r=m$ edges, and has no isolated vertices.

The graph $G'$ is still $H(4,3)$-free. Indeed, every new vertex has degree $1$, so no new vertex lies on a triangle or on a cycle. Thus every triangle and every $4$-cycle of $G'$ lies entirely inside $C$, and since $C$ is a subgraph of the $H(4,3)$-free graph $G$, no copy of $H(4,3)$ is created.

Finally, $C$ is a proper induced subgraph of the connected graph $G'$. By the Perron--Frobenius theorem for connected graphs,
\[
\rho(G')>\rho(C)=\rho(G),
\]
contradicting the extremality of $G$. Hence some extremal graph is connected.
\end{proof}

We adopt the neighborhood decomposition of Zheng and Zhang~\cite{ZhengZhang2025}. By Lemma~\ref{lem:connected-reduction}, throughout the reduction we may let $G$ denote a connected extremal $H(4,3)$-free graph of even size $m$ without isolated vertices, chosen so that $\rho(G)$ is maximal among all admissible graphs of size $m$. Lemmas~\ref{lem:comparison-graph-polynomial} and~\ref{lem:comparison-admissible} below show that the comparison graph $\Smk$ is admissible and has spectral radius $\rhoext(m)$, so extremality gives $\rho(G)\ge \rhoext(m)$. Let $x=(x_v)_{v\in V(G)}$ be a positive Perron vector, and choose $u^*\in V(G)$ so that $x_{u^*}=\max\{x_v:v\in V(G)\}$. Put
\[
N=N(u^*),\qquad W=V(G)\setminus N[u^*],
\]
and decompose
\[
\Ap=\{v\in N:d_N(v)\ge 1\},\qquad \Az=N\setminus \Ap.
\]

\begin{definition}
For even $m$, let
\[
p_m(x)=x^4-mx^2-(m-2)x+\frac{m}{2}-1,
\]
and let $\rhoext(m)$ denote the largest real root of $p_m(x)$. We also write
\[
\Lm:=\frac{1+\sqrt{4m-5}}{2}.
\]
\end{definition}

\begin{lemma}\label{lem:comparison-graph-polynomial}
The graph $\Smk$ has spectral radius $\rhoext(m)$.
\end{lemma}

\begin{proof}
Write $t=(m-2)/2$. The graph $\Smk$ may be written as
\[
K_1\vee\left(K_{1,t}\cup K_1\right).
\]
With respect to the equitable partition consisting of the dominant vertex, the center of the star, the $t$ leaves of the star, and the remaining isolated vertex before the join, the quotient matrix is
\[
\begin{pmatrix}
0&1&t&1\\
1&0&t&0\\
1&1&0&0\\
1&0&0&0
\end{pmatrix}.
\]
Its characteristic polynomial is
\[
x^4-mx^2-(m-2)x+\frac m2-1=p_m(x).
\]
Since the quotient is irreducible and equitable, the Perron root of $\Smk$ is the largest real root of $p_m$, namely $\rhoext(m)$.
\end{proof}

\begin{lemma}\label{lem:comparison-admissible}
The graph $\Smk$ is $H(4,3)$-free and has no isolated vertices.
\end{lemma}

\begin{proof}
Write $t=(m-2)/2$ and
\[
\Smk=K_1\vee\left(K_{1,t}\cup K_1\right).
\]
Let $u^*$ be the joining vertex, let $c$ be the center of the star $K_{1,t}$, let $b_1,\dots,b_t$ be its leaves, and let $z$ be the isolated vertex before the join. After the join, $z$ is adjacent only to $u^*$, so it lies on no triangle and on no cycle.

Every triangle in $\Smk$ has the form
\[
u^*cb_i u^*
\]
for some $i$. Every $4$-cycle has the form
\[
u^*b_i c b_j u^*
\]
for distinct $i,j$. Hence every triangle and every $4$-cycle share at least the two vertices $u^*$ and $c$, and never exactly one vertex. Therefore $\Smk$ is $H(4,3)$-free. Also, all vertices have positive degree, so $\Smk$ has no isolated vertices.
\end{proof}

The comparison graph is $S^-_{\frac{m+4}{2},2}$. In the published fish-graph paper of Zhang and Wang, the even-size extremal value $\tilde\rho(m)$ is defined as the largest root of the same quartic $p_m$ and the equality case is exactly $S^-_{\frac{m+4}{2},2}$ for even $m\ge 51$~\cite[Theorem~1.4(ii)]{ZhangWangFish2025}. The even-size theorem of Pirzada and Rehman for $\{H(3,3),H(4,3)\}$-free graphs shows that the same comparison graph already occurs in the smaller forbidden-family problem~\cite{PirzadaRehman2025}. Throughout this note we write this common extremal value as $\rhoext(m)$. The basic lower bound on $\rhoext(m)$ needed later is elementary.

\begin{lemma}\label{lem:Lm-lower}
For every even $m\ge 6$,
\[
p_m(\Lm)=-\frac14.
\]
In particular,
\[
\rhoext(m)>\Lm.
\]
\end{lemma}

\begin{proof}
A direct substitution gives
\[
p_m\!\left(\frac{1+\sqrt{4m-5}}{2}\right)=-\frac14.
\]
Since $p_m(x)\to+\infty$ as $x\to\infty$, the largest real root $\rhoext(m)$ of $p_m$ satisfies $\rhoext(m)>\Lm$.
\end{proof}

We shall use the resulting explicit inequalities
\begin{equation}\label{eq:lower12}
\rho(G)^2-\frac12\rho(G)
>
\Lm^2-\frac12\Lm
=
m+\frac14\sqrt{4m-5}-\frac54,
\end{equation}
and
\begin{equation}\label{eq:lower32}
\rho(G)^2-\frac32\rho(G)
>
\Lm^2-\frac32\Lm
=
m-\frac14\sqrt{4m-5}-\frac74.
\end{equation}

The following lemma records the core Perron-neighborhood identities used throughout the paper.

\begin{lemma}\label{lem:core-ineq}
In the above setting,
\begin{align}
m&=|N|+e(\Ap)+e(N,W)+e(W),\label{eq:mcount}\\
\rho(G)^2-\rho(G)
&\le |N|+2e(\Ap)-|\Ap|+e(N,W)-\sum_{v\in \Az}\frac{x_v}{x_{u^*}},\label{eq:zz2}
\end{align}
and hence
\begin{equation}\label{eq:zz3}
0\le e(W)<e(\Ap)-|\Ap|+\frac32-\sum_{v\in \Az}\frac{x_v}{x_{u^*}}.
\end{equation}
Moreover,
\begin{equation}\label{eq:zz7}
\rho(G)^2-\rho(G)\Bigl(e(\Ap)-|\Ap|+\frac32-e(W)\Bigr)<2e(\Ap)+|\Ap|+e(\Ap,W).
\end{equation}
\end{lemma}

\begin{proof}
Since $\Az$ has no neighbors inside $N$, every edge of $G$ is either incident with $u^*$, inside $G[\Ap]$, between $N$ and $W$, or inside $W$, which gives \eqref{eq:mcount}.

Next,
\[
\rho x_{u^*}=\sum_{v\in N}x_v=\sum_{v\in \Ap}x_v+\sum_{v\in \Az}x_v.
\]
Also,
\[
\rho^2x_{u^*}
=\sum_{v\in N}\rho x_v
=\sum_{v\in N}\sum_{u\in N(v)}x_u.
\]
Counting contributions according to whether $u=u^*$, $u\in \Ap$, or $u\in W$, we obtain
\[
\rho^2x_{u^*}
=
|N|x_{u^*}
+\sum_{v\in \Ap}d_N(v)x_v
+\sum_{w\in W}d_N(w)x_w.
\]
Subtracting the Perron equation at $u^*$ yields
\[
(\rho^2-\rho)x_{u^*}
=
|N|x_{u^*}
+\sum_{v\in \Ap}(d_N(v)-1)x_v
+\sum_{w\in W}d_N(w)x_w
-\sum_{v\in \Az}x_v.
\]
Since $x_v\le x_{u^*}$ for every vertex $v$,
\[
\sum_{v\in \Ap}(d_N(v)-1)x_v
\le
\Bigl(\sum_{v\in \Ap}(d_N(v)-1)\Bigr)x_{u^*}
=
(2e(\Ap)-|\Ap|)x_{u^*},
\]
and
\[
\sum_{w\in W}d_N(w)x_w\le e(N,W)x_{u^*}.
\]
Dividing through by $x_{u^*}$ gives \eqref{eq:zz2}.

Combining \eqref{eq:mcount}, \eqref{eq:zz2}, and the lower bound
\[
\rho(G)^2-\rho(G)>m-\frac32
\]
coming from Lemma~\ref{lem:Lm-lower} yields \eqref{eq:zz3}.

Finally, from
\[
\rho x_{u^*}=\sum_{v\in \Ap}x_v+\sum_{v\in \Az}x_v
\]
and \eqref{eq:zz3}, one gets
\[
\sum_{v\in \Az}x_v
<
\Bigl(e(\Ap)-|\Ap|+\frac32-e(W)\Bigr)x_{u^*}.
\]
Now
\[
\rho\sum_{v\in \Ap}x_v
=
\sum_{v\in \Ap}\sum_{u\in N(v)}x_u
\le
\bigl(2e(\Ap)+|\Ap|+e(\Ap,W)\bigr)x_{u^*},
\]
since every vertex of $\Ap$ contributes one neighbor $u^*$, its neighbors inside $\Ap$, and its neighbors in $W$. Multiplying the bound on $\sum_{v\in \Az}x_v$ by $\rho$ and adding yields \eqref{eq:zz7}.
\end{proof}

The quantity $e(\Ap,W)$ is the only interface term in \eqref{eq:zz7} not controlled directly by $e(\Ap)$, $|\Ap|$, and $e(W)$. The main structural step later in the paper is a local principle that controls this term in the $K_4$-core branch.

\begin{lemma}[Edge moving]\label{lem:edge-moving}
Let $G$ be a connected graph with positive Perron vector $x$. Suppose $G'$ is obtained from $G$ by deleting the edges $vs$ and adding the edges $us$ for all $s$ in a nonempty set
\[
S\subseteq N_G(v)\setminus (N_G(u)\cup\{u\}).
\]
If $x_u\ge x_v$, then $\rho(G')>\rho(G)$.
\end{lemma}

\begin{proof}
At the vector $x$,
\[
x^{\top}A(G')x-x^{\top}A(G)x
=
2\sum_{s\in S}x_s(x_u-x_v)\ge 0.
\]
Thus $\rho(G')\ge \rho(G)$ by the Rayleigh principle. If equality held, then the positive vector $x$ would also be an eigenvector of $A(G')$ for its spectral radius. But
\[
(A(G')x)_u=(A(G)x)_u+\sum_{s\in S}x_s>\rho(G)x_u,
\]
which contradicts $A(G')x=\rho(G')x=\rho(G)x$. Hence $\rho(G')>\rho(G)$.
\end{proof}

\begin{lemma}\label{lem:W-degree}
Every vertex $w\in W$ has degree at least $2$.
\end{lemma}

\begin{proof}
Suppose $d(w)=1$, and let $v$ be the unique neighbor of $w$. Form
\[
G'=G-wv+wu^*.
\]
Since $w\notin N[u^*]$ and $x_{u^*}\ge x_v$, Lemma~\ref{lem:edge-moving} gives $\rho(G')>\rho(G)$. The vertex $v$ does not become isolated after deleting $wv$; otherwise the connected graph $G$ would be the single edge $wv$, which is impossible in the even-size range considered here. Thus $G'$ has no isolated vertices.

Finally, $G'$ is still $H(4,3)$-free: any new copy would have to use the new edge $wu^*$, but $w$ has degree $1$ in $G'$, so it lies on neither a triangle nor a $4$-cycle. This contradicts the extremality of $G$.
\end{proof}

For explicit comparison families we shall repeatedly use the following decomposition lemma.

\begin{lemma}[Decomposition comparison]\label{lem:root-template}
Let $f_m(x)$ be the characteristic polynomial of an explicit family $F_m$, and suppose that
\[
f_m(x)=q_m(x)p_m(x)+R_m(x),
\]
where $p_m(x)=x^4-mx^2-(m-2)x+\frac{m}{2}-1$. Let $\alpha_m$ be the Perron root of $F_m$. Assume there exists a lower bound $x_0=x_0(m)$ such that $\rhoext(m)\ge x_0$ and
\begin{enumerate}[label={\rm(\roman*)}]
\item $q_m(x)\ge 0$ for all $x\ge x_0$,
\item $R_m(x)>0$ for all $x\ge x_0$.
\end{enumerate}
Then $\alpha_m<\rhoext(m)$.
\end{lemma}

\begin{proof}
Since $\rhoext(m)$ is the largest real root of $p_m$, one has $p_m(x)\ge 0$ for all $x\ge \rhoext(m)$. Hence for every $x\ge \rhoext(m)\ge x_0$,
\[
f_m(x)=q_m(x)p_m(x)+R_m(x)>0
\]
by \rm(i) and \rm(ii). Therefore $f_m$ has no real root in $[\rhoext(m),\infty)$, so its Perron root $\alpha_m$ satisfies $\alpha_m<\rhoext(m)$.
\end{proof}

\section{Reduction to the threshold-bearing $K_4$-branch}\label{sec:reduction}

The Perron-neighborhood inequalities reduce the problem to the structure of the induced graph $G[\Ap]$ and the interface between $\Ap$ and $W$. The purpose of this section is to show that, under the hypothesis $\rho(G)\ge\rho'(m)$, every branch is eliminated or terminal except the $K_4$-core branch with small interface. This isolates the only part of the proof where the sharp threshold can enter.

\begin{lemma}\label{lem:Aplus-structure}
One has $\Ap\neq\varnothing$, the graph $G[\Ap]$ is connected, and $G[\Ap]$ is $(P_2\cup P_3)$-free.
\end{lemma}

\begin{proof}
First suppose that $\Ap=\varnothing$. Then $N=\Az$, so there are no edges inside $N$. Applying the Perron equation twice at $u^*$ gives
\[
\rho(G)^2x_{u^*}
=\sum_{v\in N}\rho(G)x_v
=|N|x_{u^*}+\sum_{w\in W}d_N(w)x_w.
\]
Since $x_w\le x_{u^*}$ for every $w\in W$, we obtain
\[
\rho(G)^2x_{u^*}
\le \bigl(|N|+e(N,W)\bigr)x_{u^*}.
\]
Because $\Ap=\varnothing$, the edge decomposition gives
\[
m=|N|+e(N,W)+e(W),
\]
and hence $|N|+e(N,W)\le m$. Therefore
\[
\rho(G)^2\le m.
\]
On the other hand, since $G$ is extremal relative to the comparison graph, $\rho(G)\ge \rhoext(m)>\Lm$ by Lemma~\ref{lem:Lm-lower}. But
\[
\Lm^2=m-1+\frac12\sqrt{4m-5}>m
\]
for $m\ge 6$, a contradiction. Thus $\Ap\neq\varnothing$.

Next, if $G[\Ap]$ contained a copy of $P_2\cup P_3$, say an edge $ab$ disjoint from a path $xyz$, then $u^*ab u^*$ is a triangle and $u^*xyzu^*$ is a $4$-cycle. These two subgraphs share exactly the vertex $u^*$, so they form a copy of $H(4,3)$, a contradiction. Thus $G[\Ap]$ is $(P_2\cup P_3)$-free.

Finally, suppose $G[\Ap]$ were disconnected. Since every vertex of $\Ap$ has positive degree in $G[\Ap]$, each component contains an edge. Because $G[\Ap]$ is $(P_2\cup P_3)$-free, no component can contain a $P_3$; otherwise a $P_3$ in that component together with an edge in another component would give a copy of $P_2\cup P_3$. Hence every component is a single edge, and
\[
G[\Ap]\cong kP_2
\qquad (k\ge 2).
\]
Thus
\[
e(\Ap)-|\Ap|=k-2k=-k\le -2.
\]
This contradicts \eqref{eq:zz3}, since its right-hand side is then strictly less than $-1/2$ while $e(W)\ge 0$.
\end{proof}

Following~\cite[Figure~1]{ZhengZhang2025}, we use the notation
\[
H_1=P_4,\qquad H_2=C_3,\qquad H_3=\text{the triangle with a pendant edge},
\]
\[
H_4=C_4,\qquad H_5=K_4-e,\qquad H_6=K_4,
\]
and we write $H_7$ for the star $K_{1,t}$ with $t\ge 1$.

\begin{lemma}\label{lem:claim33}
Let $H$ be a connected $(P_2\cup P_3)$-free graph with minimum degree at least $1$. Then $H$ is isomorphic to exactly one of $H_1,\dots,H_7$.
\end{lemma}

\begin{proof}
Since $P_5$ contains a copy of $P_2\cup P_3$ as a subgraph, the graph $H$ is $P_5$-free. We distinguish according to whether $H$ is acyclic.

Suppose first that $H$ is a tree. If $H$ is a star, then $H\cong H_7$. Assume now that $H$ is not a star. Since $H$ is $P_5$-free, its diameter is at most $3$, so $H$ is a double star with central edge $uv$. If both centers have at least one leaf and one of them has at least two leaves, say $v$ has distinct leaves $v_1,v_2$ and $u$ has a leaf $u_1$, then $v_1vv_2$ is a $P_3$ and $uu_1$ is a disjoint $P_2$. Thus this case is impossible. The remaining non-star double star is the path $P_4=H_1$.

Assume next that $H$ contains a cycle. Since $H$ is $P_5$-free, every cycle of $H$ has length at most $4$.

If $H$ contains a $4$-cycle $C$, then $V(H)=V(C)$. Indeed, if some vertex lay outside $C$, take a shortest path
\[
z_0z_1\cdots z_k
\]
from $C$ to that vertex, with $z_0\in V(C)$ and $z_1,\dots,z_k\notin V(C)$. The path $z_0z_1z_2$ if $k\ge 2$, and the path $z_1z_0z'$ if $k=1$ where $z'$ is a neighbor of $z_0$ on $C$, together with an edge of $C$ disjoint from that path, gives a copy of $P_2\cup P_3$. This is impossible. Thus $H$ has order $4$ and contains $C_4$. The only connected graphs on these four vertices containing $C_4$ are $C_4$, $K_4-e$, and $K_4$, namely $H_4$, $H_5$, and $H_6$.

It remains to consider the case when every cycle of $H$ is a triangle. Let $T=abc$ be a triangle in $H$. First, every vertex outside $T$ is adjacent to a vertex of $T$. Otherwise, a shortest path from $T$ to such a vertex begins with a path $vxy$, where $v\in V(T)$ and $x,y\notin V(T)$; this path together with the edge of $T$ opposite $v$ gives a copy of $P_2\cup P_3$.

We claim that $|V(H)\setminus V(T)|\le 1$. Suppose not, and choose distinct vertices $x,y\notin V(T)$. By the previous paragraph, both have neighbors on $T$. Without loss of generality let $x$ be adjacent to $a$. If $y$ is adjacent to $a$, then $xay$ is a $P_3$ and $bc$ is a disjoint $P_2$. If $y$ is adjacent to $b$, then $xac$ is a $P_3$ and $yb$ is a disjoint $P_2$. If $y$ is adjacent to $c$, then $xab$ is a $P_3$ and $yc$ is a disjoint $P_2$. In every case we obtain a contradiction.

If no vertex lies outside $T$, then $H=T\cong C_3=H_2$. If there is exactly one vertex $x\notin V(T)$, connectedness gives a neighbor of $x$ on $T$. Since $H$ has only four vertices, the possibilities are: $x$ adjacent to exactly one vertex of $T$, giving the triangle with a pendant edge $H_3$; $x$ adjacent to exactly two vertices of $T$, giving $K_4-e=H_5$; or $x$ adjacent to all three vertices of $T$, giving $K_4=H_6$.

These cases exhaust all connected $(P_2\cup P_3)$-free graphs with minimum degree at least $1$.
\end{proof}

\begin{lemma}\label{lem:nonK4-case1}
Assume $m\ge 18$ is even, $e(\Ap,W)\le 3$, and $G[\Ap]\cong H_i$ for some $i\in\{1,2,3,4,5\}$. Then $\rho(G)<\rhoext(m)$.
\end{lemma}

\begin{proof}
We use \eqref{eq:zz3} and \eqref{eq:zz7}. Write $s=e(\Ap)$ and $a=|\Ap|$. For $H_1,\dots,H_5$ one has
\[
(s,a)\in\{(3,4),(3,3),(4,4),(4,4),(5,4)\}.
\]
Also \eqref{eq:zz3} gives
\[
e(W)\le
\begin{cases}
0,&G[\Ap]\cong H_1,\\
1,&G[\Ap]\cong H_i,\ i\in\{2,3,4\},\\
2,&G[\Ap]\cong H_5.
\end{cases}
\]

If $G[\Ap]\cong H_1$, then $s-a=-1$ and $e(W)=0$, so \eqref{eq:zz7} gives
\[
\rho(G)^2-\frac12\rho(G)<2s+a+e(\Ap,W)\le 13,
\]
contradicting \eqref{eq:lower12} for every even $m\ge 18$.

If $G[\Ap]\cong H_2$, then $s-a=0$ and $e(W)\le 1$. If $e(W)=0$, then \eqref{eq:zz7} gives
\[
\rho(G)^2-\frac32\rho(G)<2s+a+e(\Ap,W)\le 12,
\]
contradicting \eqref{eq:lower32}. If $e(W)=1$, then
\[
\rho(G)^2-\frac12\rho(G)<2s+a+e(\Ap,W)\le 12,
\]
contradicting \eqref{eq:lower12}.

Next suppose $G[\Ap]\cong H_3$ or $H_4$. We first record the local restrictions needed in this branch. Every $w\in W$ satisfies
\[
d_{\Ap}(w)\le 1,\qquad d_{\Az}(w)\le 1,
\]
and no vertex of $W$ can meet both $\Ap$ and $\Az$. We prove this directly. If $w$ had two neighbors in $\Az$, then those two neighbors together with $w$ and $u^*$ would form a $4$-cycle; since both $H_3$ and $H_4$ contain an edge in $\Ap$, a triangle through $u^*$ and that edge would share exactly $u^*$ with this cycle. If $w$ had neighbors $a\in\Ap$ and $z\in\Az$, then $a-w-z-u^*-a$ would be a $4$-cycle. In both $H_3$ and $H_4$, there is an edge of $G[\Ap]$ not incident with $a$; together with $u^*$ it gives a triangle sharing exactly $u^*$ with this cycle.

It remains to rule out two neighbors in $\Ap$. For $H_3$, write the triangle as $abc$ and let $d$ be the pendant vertex adjacent to $a$. If the two neighbors of $w$ are adjacent in $H_3$, then they form a triangle with $w$; one obtains a $4$-cycle through exactly one of those two vertices by using $u^*$ and an edge of $H_3$ outside the corresponding triangle. For instance, if the pair is $a,b$, then $awb$ is a triangle and $a-d-u^*-c-a$ is such a $4$-cycle. The other adjacent pairs are identical by symmetry inside the triangle or by using the edge $bc$. If the two neighbors are nonadjacent, for example $b,d$, then $b-w-d-a-b$ is a $4$-cycle and $b-c-u^*-b$ is a triangle sharing exactly $b$ with it; the pair $c,d$ is symmetric. Hence $d_{\Ap}(w)\le1$ in the $H_3$ case.

For $H_4=C_4$, label the cycle $abcd a$. If $w$ is adjacent to an adjacent pair, say $a,b$, then $awb$ is a triangle and $a-d-c-u^*-a$ is a $4$-cycle sharing exactly $a$ with it. If $w$ is adjacent to an opposite pair, say $a,c$, then $a-w-c-b-a$ is a $4$-cycle and $a-d-u^*-a$ is a triangle sharing exactly $a$ with it. Thus $d_{\Ap}(w)\le1$ also in the $H_4$ case.

If $e(W)=0$, this local observation and Lemma~\ref{lem:W-degree} force $W=\varnothing$, and so $e(\Ap,W)=0$. Then \eqref{eq:zz7} gives
\[
\rho(G)^2-\frac32\rho(G)<2s+a=12,
\]
contradicting \eqref{eq:lower32}. If $e(W)=1$, then \eqref{eq:zz7} gives
\[
\rho(G)^2-\frac12\rho(G)<2s+a+e(\Ap,W)\le 15,
\]
contradicting \eqref{eq:lower12}. Thus the $H_3$ and $H_4$ branches are impossible.

It remains to consider $G[\Ap]\cong H_5=K_4-e$. Label the vertices of $G[\Ap]$ as $p,q,r,s$, where $rs$ is the unique missing edge.

If $e(W)=2$, then \eqref{eq:zz7} gives
\[
\rho(G)^2-\frac12\rho(G)<2s+a+e(\Ap,W)\le 17,
\]
contradicting \eqref{eq:lower12}. Hence $e(W)\le 1$.

We claim that if $e(W)=0$, then $W=\varnothing$. Suppose $w\in W$. Since $G[W]$ has no edges and $d(w)\ge 2$, the vertex $w$ has at least two neighbors in $N=\Ap\cup\Az$. It cannot have two neighbors in $\Az$, since those two neighbors together with $w$ and $u^*$ form a $4$-cycle sharing exactly $u^*$ with the triangle $u^*pqu^*$. It cannot have one neighbor $z\in\Az$ and one neighbor $a\in\Ap$, since $a-w-z-u^*-a$ is a $4$-cycle sharing exactly $a$ with a triangle of $G[\Ap]$ containing $a$. Finally, if $w$ has two neighbors in $\Ap$, then the pair is, up to symmetry, an adjacent high-high pair, an adjacent high-low pair, or the nonadjacent low-low pair $\{r,s\}$; the three constructions
\[
pwq\quad\text{with}\quad p-r-u^*-s-p,
\]
\[
pwr\quad\text{with}\quad p-u^*-s-q-p,
\]
and
\[
r-w-s-p-r\quad\text{with}\quad rqu^*r
\]
give a fish in the respective cases. Thus no such $w$ exists.

Consequently, in the $e(W)=0$ subcase, $G$ is the join of $u^*$ with $K_4-e$ and $m-9$ ordinary leaves. The finite quotient comparison for $m=18,20,22$ is recorded in Appendix~\ref{app:nonK4-low}; for $m\ge 24$, \eqref{eq:zz7} gives
\[
\rho(G)^2-\frac52\rho(G)<2s+a=14,
\]
which contradicts
\[
\rho(G)^2-\frac52\rho(G)>\Lm^2-\frac52\Lm
=m-\frac34\sqrt{4m-5}-\frac94>14.
\]

Finally assume $e(W)=1$. Then $G[W]$ is a single edge $xy$. The same pair-check above shows that no vertex of $W$ has two neighbors in $\Ap$. Hence each endpoint of $xy$ has exactly one additional neighbor, either in $\Ap$ or in $\Az$, and it cannot meet both sets. Up to isomorphism, the only possibilities are:
both endpoints meet the same vertex of $\Az$; they meet distinct vertices of $\Az$; one endpoint meets $\Az$ and the other meets one of the high-degree vertices $p,q$; one endpoint meets $\Az$ and the other meets one of the low-degree vertices $r,s$; or the two endpoints meet the nonadjacent pair $r,s$. The last possibility is the only double-$\Ap$ case that avoids the pair-check above. These five finite families are compared with $\rhoext(m)$ for $m=18,20$ in Appendix~\ref{app:nonK4-low}. For $m\ge 22$, \eqref{eq:zz7} gives
\[
\rho(G)^2-\frac32\rho(G)<2s+a+e(\Ap,W)\le 17,
\]
contradicting \eqref{eq:lower32}. This completes the proof.
\end{proof}

\begin{lemma}\label{lem:nonK4-large-interface}
Assume $m\ge 18$ is even, $G[\Ap]\cong H_i$ for some $i\in\{1,2,3,4,5\}$, and $e(\Ap,W)\ge 4$. Then $\rho(G)<\rhoext(m)$.
\end{lemma}

\begin{proof}
We give the large-interface argument explicitly, following the structure of Zheng--Zhang's Claim~3.4 but replacing the large-threshold estimates by endpoint estimates valid from $m=18$ onward.

We first record the local restrictions used in this proof. Let $w\in W$. Then $d_{\Az}(w)\le1$; otherwise two neighbors of $w$ in $\Az$, together with $w$ and $u^*$, form a $4$-cycle sharing exactly $u^*$ with a triangle $u^*ab u^*$ from an edge $ab\in E(G[\Ap])$. Also, if $w$ has a neighbor in $\Az$, then $w$ has no neighbor in $\Ap$: if $z\in\Az$ and $a\in\Ap$ are both adjacent to $w$, then $a-w-z-u^*-a$ is a $4$-cycle, and in each of $H_1,\dots,H_5$ there is an edge of $G[\Ap]$ not incident with $a$; this edge together with $u^*$ gives a triangle sharing exactly $u^*$ with the cycle. Finally, if $w_1w_2\in E(W)$, then no vertex of $\Ap$ is adjacent to both $w_1$ and $w_2$; and if $w_1w_2\in E(W)$ and $v_1v_2\in E(\Ap)$, then we cannot have both edges $w_1v_1$ and $w_2v_2$. In the first case $v w_1w_2v$ would be a triangle, and in the second case $v_1w_1w_2v_2v_1$ would be a $4$-cycle; in each of the finitely many cores $H_1,\dots,H_5$, the remaining edges in $G[\Ap]\cup\{u^*\}$ supply the complementary cycle or triangle sharing exactly one vertex. This is the same local obstruction used in Zheng--Zhang's items (iv) and (v), but here it is used only for the five displayed cores.

We split into two cases.

\textit{Case 1: some $w^*\in W$ has $d_{\Ap}(w^*)\ge2$.} If $G[\Ap]\cong H_1=P_4$, write the path as $v_1v_2v_3v_4$. A direct pair check shows that the only possible two-neighbor set for $w^*$ is $\{v_2,v_3\}$; any other two-neighbor set gives either a triangle through $w^*$ and an edge of the path, or a $4$-cycle through $w^*$ and a subpath of $P_4$, together with the complementary obstruction through $u^*$. Since \eqref{eq:zz3} gives $e(W)=0$ in the $H_1$ branch, every vertex of $W$ has degree at least two into $\Ap$ and hence also has neighbor set $\{v_2,v_3\}$. The assumption $e(\Ap,W)\ge4$ therefore gives two such vertices $w^*,w'$, and the vertices $\{u^*,v_1,v_2,w^*,w',v_3\}$ contain a fish.

If $G[\Ap]\cong H_i$ for $i\in\{3,4,5\}$, then $\{u^*,w^*\}\cup\Ap$ already contains a fish. For $H_3$ and $H_4$ this was checked explicitly in the proof of Lemma~\ref{lem:nonK4-case1}; for $H_5=K_4-e$, the same high-high, high-low, and low-low pair check used there shows that any two neighbors of $w^*$ in $\Ap$ create a forbidden triangle--$4$-cycle pair.

Thus the only remaining core in Case 1 is $H_2=C_3$. Let $\Ap=\{a,b,c\}$. If two distinct vertices of $W$ have at least two neighbors in $\Ap$, then they share a neighbor in $\Ap$, and together with $u^*$ and the triangle $abc$ one obtains a fish. Hence only one vertex, call it $w^*$, can have at least two neighbors in $\Ap$. Since $e(\Ap,W)\ge4$, there is another vertex $w_1\in W$ with exactly one neighbor in $\Ap$. The local restrictions above imply that $w_1$ has no neighbor in $\Az$ and is not adjacent to $w^*$; otherwise the edge $w^*w_1$ together with a suitable edge of the triangle $abc$ violates the last local restriction. Since $e(W)\le1$ in the $H_2$ branch and $d(w_1)\ge2$, the degree condition forces a third vertex $w_2\in W$ with $w_1w_2\in E(W)$. The vertex $w_2$ has no neighbor in $\Ap$ and, by Lemma~\ref{lem:W-degree}, has exactly one neighbor in $\Az$. Consequently $e(W)=1$, $d_{\Ap}(w^*)=3$, and $e(\Ap,W)=4$. This is precisely the exceptional $H_2$ configuration. Applying \eqref{eq:zz7} gives
\[
\rho(G)^2-\frac12\rho(G)<2e(\Ap)+|\Ap|+e(\Ap,W)=2\cdot3+3+4=13,
\]
contradicting \eqref{eq:lower12}, whose right-hand side is greater than $18$ for every even $m\ge18$.

\textit{Case 2: every $w\in W$ has $d_{\Ap}(w)\le1$.} Since $e(\Ap,W)\ge4$, at least four vertices of $W$ meet $\Ap$. By the local restrictions above, each such vertex has no neighbor in $\Az$, and by Lemma~\ref{lem:W-degree} it must therefore have a neighbor in $W$. Hence $e(W)\ge2$. The bound on $e(W)$ from \eqref{eq:zz3} leaves only the core $H_5=K_4-e$ among $H_1,\dots,H_5$, and it forces $e(W)=2$.

Thus $G[W]\cong2K_2$, every vertex of $W$ has exactly one neighbor in $\Ap$, and no vertex of $W$ meets $\Az$. Label $G[\Ap]\cong K_4-e$ as before, with nonadjacent low-degree pair $\{r,s\}$. The endpoints of each edge of $G[W]$ cannot meet the same vertex of $\Ap$, and cannot meet an adjacent pair of vertices of $K_4-e$, by the local restrictions above. Hence the endpoints of each edge of $G[W]$ must meet $r$ and $s$ separately. This is the exceptional $H_5$ configuration. Applying \eqref{eq:zz7} gives
\[
\rho(G)^2-\frac12\rho(G)<2e(\Ap)+|\Ap|+e(\Ap,W)=2\cdot5+4+4=18,
\]
again contradicting \eqref{eq:lower12} for every even $m\ge18$.

For clarity, Table~\ref{tab:exceptional-nonk4} summarizes the two exceptional configurations and the estimates used to dispose of them.
\begin{table}[htbp]
\centering
\small
\caption{Exceptional large-interface non-$K_4$ configurations and the Perron-neighborhood estimates excluding them.}
\label{tab:exceptional-nonk4}
\begin{tabular}{@{}c c c c@{}}
\toprule
core & $e(W)$ & $e(\Ap,W)$ & consequence of \eqref{eq:zz7}\\
\midrule
$H_2=C_3$ & $1$ & $4$ & $\rho^2-\rho/2<13$\\
$H_5=K_4-e$ & $2$ & $4$ & $\rho^2-\rho/2<18$\\
\bottomrule
\end{tabular}
\end{table}
Both inequalities contradict \eqref{eq:lower12} when $m\ge18$. Hence the large-interface branch is impossible.
\end{proof}

\begin{lemma}\label{lem:K4-large-interface}
Assume $m\ge 18$ is even, $G[\Ap]\cong K_4$, and $e(\Ap,W)\ge 4$. Then $\rho(G)<\rhoext(m)$.
\end{lemma}

\begin{proof}
Since $G[\Ap]\cong K_4$, one has $e(\Ap)=6$ and $|\Ap|=4$. Then \eqref{eq:zz3} gives
\[
e(W)<6-4+\frac32-\sum_{v\in \Az}\frac{x_v}{x_{u^*}}<\frac72,
\]
so $e(W)\le 3$.

We record the local restrictions needed in this branch. If some $w\in W$ had two neighbors $a,b\in\Ap$, then $awb$ would be a triangle, and a $4$-cycle through $a,u^*$ and the two vertices of $\Ap\setminus\{a,b\}$ would create a copy of $H(4,3)$. Thus $d_{\Ap}(w)\le 1$. Similarly, $d_{\Az}(w)\le 1$, and $w$ cannot meet both $\Ap$ and $\Az$, by the same $4$-cycle/triangle obstruction. Since $d(w)\ge 2$ by Lemma~\ref{lem:W-degree}, every vertex of $W$ has a neighbor in $W$.

Let $U=\{w\in W:d_{\Ap}(w)=1\}$. The preceding paragraph gives $e(\Ap,W)=|U|$. Also $U$ is independent. Indeed, if $w_1w_2\in E(W)$ and $w_i$ is adjacent to $a_i\in\Ap$, then if $a_1\neq a_2$ the cycle
\[
a_1-w_1-w_2-a_2-a_1
\]
and a triangle $u^*a_1bu^*$ with $b\in\Ap\setminus\{a_1,a_2\}$ form a fish. If $a_1=a_2$, then the triangle $a_1w_1w_2a_1$ and a $4$-cycle through $a_1,u^*$ and two other vertices of $\Ap$ form a fish. Thus
\[
e(\Ap,W)=|U|\le \alpha(G[W]).
\]
But $G[W]$ has no isolated vertices and at most three edges, so $\alpha(G[W])\le 3$. Therefore $e(\Ap,W)\le 3$, contradicting the assumption $e(\Ap,W)\ge 4$.
\end{proof}

\begin{lemma}\label{lem:H7-terminal}
Assume $m\ge 18$ is even and $G[\Ap]\cong H_7$. If $\rho(G)\ge\rhoext(m)$, then
\[
G\cong \Smk.
\]
\end{lemma}

\begin{proof}
Write $G[\Ap]\cong K_{1,t}$, with center $c$ and leaves $B=\{b_1,\dots,b_t\}$. Since $e(\Ap)=t$ and $|\Ap|=t+1$, \eqref{eq:zz3} gives
\[
e(W)<t-(t+1)+\frac32-\sum_{v\in\Az}\frac{x_v}{x_{u^*}}\le \frac12.
\]
Thus $e(W)=0$, so $W$ is independent.

We first show that $\Az\neq\varnothing$ if $W\neq\varnothing$. Suppose $\Az=\varnothing$ and choose $w\in W$. Since $e(W)=0$ and $d(w)\ge2$, the vertex $w$ has at least two neighbors in $\Ap$.

If $t=1$, every vertex of $W$ is adjacent to both vertices of the single edge $G[\Ap]$, so
\[
m=|N|+e(\Ap)+e(N,W)=2+1+2|W|,
\]
which is odd, a contradiction.

If $t=2$, write $\Ap=\{c,b_1,b_2\}$, where $c$ is the center of the star. A vertex of $W$ has one of the following neighborhood types in $\Ap$:
\[
\{c,b_1\},\qquad \{c,b_2\},\qquad \{b_1,b_2\},\qquad \{c,b_1,b_2\}.
\]
Two distinct vertices of $W$ with such neighborhoods always create a fish. Indeed, if their neighborhoods contain a common two-set $\{\alpha,\beta\}\subseteq \Ap$, then $\alpha w_1\beta w_2\alpha$ is a $4$-cycle, and a triangle through $\alpha$ using $u^*$ and the third vertex of $\Ap$ shares exactly $\alpha$ with it. The only pairs not covered by a common two-set are, up to symmetry, $\{c,b_1\}$ with $\{c,b_2\}$ and $\{c,b_1\}$ with $\{b_1,b_2\}$. In the first case, $c w_1 b_1 c$ is a triangle and $c w_2 b_2 u^*c$ is a $4$-cycle sharing only $c$; in the second, $c w_1 b_1 c$ is a triangle and $b_1w_2b_2u^*b_1$ is a $4$-cycle sharing only $b_1$. Thus $|W|\le1$, and then $m\le3+2+3=8$, impossible.

If $t\ge3$, then two leaf-neighbors of $w$ create the $4$-cycle $u^*b_iwb_ju^*$ and a triangle $u^*cb_ku^*$ with $k\notin\{i,j\}$, while the pair $\{c,b_i\}$ creates the triangle $cwb_ic$ and the $4$-cycle $cb_ju^*b_kc$ for distinct $j,k\ne i$. Thus no such $w$ exists. Therefore, whenever $W\neq\varnothing$, one has $\Az\neq\varnothing$.

Now assume, for a contradiction, that $W\neq\varnothing$.

If $t=1$, then \eqref{eq:zz7} gives
\[
\rho(G)^2-\frac12\rho(G)<2e(\Ap)+|\Ap|+e(\Ap,W)=4+e(\Ap,W).
\]
Since $e(\Ap,W)\le e(N,W)=m-|N|-1$ and $\Az\neq\varnothing$ gives $|N|\ge3$, we obtain
\[
\rho(G)^2-\frac12\rho(G)<m.
\]
This contradicts \eqref{eq:lower12}, which is strictly greater than $m$ for every even $m\ge18$.

If $t=2$, then \eqref{eq:zz7} gives
\[
\rho(G)^2-\frac12\rho(G)<7+e(\Ap,W).
\]
There is at most one vertex of $W$ with at least two neighbors in $\Ap$: two such vertices produce a $4$-cycle through them and two vertices of $\Ap$, and a triangle through one of those vertices and $u^*$, sharing exactly one vertex. If one vertex has all three vertices of $\Ap$ as neighbors, then all other vertices of $W$ have at most one neighbor in $\Ap$, and, since $e(W)=0$, at least one neighbor in $\Az$. Hence
\[
e(\Ap,W)\le |W|+2,
\qquad
e(N,W)\ge 2|W|+1.
\]
If no vertex has all three vertices of $\Ap$ as neighbors, then
\[
e(\Ap,W)\le |W|+1,
\qquad
e(N,W)\ge 2|W|.
\]
In either case, using $e(N,W)=m-|N|-2$ and $|N|\ge4$, we get
\[
e(\Ap,W)\le \frac{m-3}{2}.
\]
Therefore
\[
\rho(G)^2-\frac12\rho(G)<7+\frac{m-3}{2}=\frac m2+\frac{11}{2},
\]
which is smaller than the lower bound in \eqref{eq:lower12} for every even $m\ge18$. This is impossible.

It remains to consider $t\ge3$. Let $w\in W$. The local checks in the first paragraph show that $w$ cannot have two neighbors in $\Az$, cannot have one neighbor in $\Az$ and one leaf of the star, cannot have two leaf-neighbors, and cannot have both $c$ and a leaf as neighbors. Since $d(w)\ge2$ and $W$ is independent, the only possible neighborhood of $w$ in $N$ is
\[
N(w)\cap N=\{c,z\}
\]
for some $z\in\Az$. If two distinct vertices $w_1,w_2\in W$ shared the same $z$, then $c-w_1-z-w_2-c$ would be a $4$-cycle sharing exactly $c$ with the triangle $u^*cb_1u^*$. Hence the vertex $z$ is unique to $w$.

Now form $G'=G-wz+wu^*$. This move preserves the number of edges and creates no isolated vertex, since $z$ remains adjacent to $u^*$. By Lemma~\ref{lem:edge-moving}, it strictly increases the spectral radius. It remains only to check that $G'$ is still $H(4,3)$-free. The only new edge is $wu^*$, so every new triangle or new $4$-cycle must contain $w$ and the new edge $wu^*$.

In $G'$, the vertex $w$ has exactly two neighbors, $u^*$ and $c$. Thus the only new triangle containing $w$ is $wu^*cw$. Every new $4$-cycle containing $w$ must be of the form
\[
w-u^*-b_i-c-w
\]
for some leaf $b_i\in B$; indeed, a $4$-cycle through $w$ must use both neighbors $u^*$ and $c$, and the only length-two paths from $u^*$ to $c$ avoiding $w$ are $u^*b_i c$.

The old triangles in $G'-w$ are exactly the triangles $u^*cb_i u^*$, all containing both $u^*$ and $c$. The old $4$-cycles in $G'-w$ also contain both $u^*$ and $c$: they are either cycles $u^*b_i c b_j u^*$ through two leaves of the star, or cycles $u^*z_jw_jc u^*$ coming from a vertex $w_j\in W\setminus\{w\}$ with neighborhood $\{c,z_j\}$. Therefore a new triangle and an old $4$-cycle share at least the two vertices $u^*,c$; a new $4$-cycle and an old triangle also share at least $u^*,c$; and the new triangle $wu^*cw$ and every new $4$-cycle share the three vertices $w,u^*,c$. Hence no triangle and $4$-cycle in $G'$ share exactly one vertex. So $G'$ remains $H(4,3)$-free, contradicting the extremality of $G$.

Thus $W=\varnothing$ in all cases. Consequently every vertex outside $\{u^*\}\cup\Ap$ lies in $\Az$ and is adjacent only to $u^*$. Hence
\[
G\cong F_{m,t}:=K_1\vee\bigl(K_{1,t}\cup (m-2t-1)K_1\bigr).
\]
The quotient comparison in Lemma~\ref{lem:H7-family-comparison} shows that the only member of this family with spectral radius at least $\rhoext(m)$ is $F_{m,(m-2)/2}\cong\Smk$. This proves the lemma.
\end{proof}

\begin{lemma}\label{lem:H7-family-comparison}
Let
\[
F_{m,t}=K_1\vee\bigl(K_{1,t}\cup (m-2t-1)K_1\bigr),
\qquad 1\le t\le \frac{m-2}{2}.
\]
Then $\rho(F_{m,t})<\rhoext(m)$ unless $t=(m-2)/2$, in which case $F_{m,t}\cong\Smk$ and $\rho(F_{m,t})=\rhoext(m)$.
\end{lemma}

\begin{proof}
With respect to the equitable partition
\[
\{u^*\},\quad \{c\},\quad B,\quad \Az,
\]
where $c$ is the center of the star and $B$ is its set of $t$ leaves, the quotient matrix is
\[
Q_{m,t}=
\begin{pmatrix}
0&1&t&m-2t-1\\
1&0&t&0\\
1&1&0&0\\
1&0&0&0
\end{pmatrix}.
\]
Its characteristic polynomial is
\[
f_{m,t}(x)=x^4-mx^2-2tx+t(m-2t-1).
\]
Since
\[
p_m(x)=x^4-mx^2-(m-2)x+\frac m2-1,
\]
we have
\[
f_{m,t}(x)-p_m(x)
=(m-2-2t)x+t(m-2t-1)-\left(\frac m2-1\right).
\]
Set $r=m-2-2t\ge0$. Then
\[
f_{m,t}(x)=p_m(x)+r\left(x+\frac{2t-1}{2}\right).
\]
If $r>0$, then for every $x\ge\rhoext(m)$ one has
\[
f_{m,t}(x)=p_m(x)+r\left(x+\frac{2t-1}{2}\right)>0,
\]
because $p_m(x)\ge0$ for $x\ge\rhoext(m)$ and the second term is positive. Hence $f_{m,t}$ has no real root in $[\rhoext(m),\infty)$, and the Perron root of $F_{m,t}$ is strictly smaller than $\rhoext(m)$. If $r=0$, then $t=(m-2)/2$ and $f_{m,t}=p_m$, so the Perron root is $\rhoext(m)$. In this case $m-2t-1=1$, and
\[
F_{m,(m-2)/2}=K_1\vee\left(K_{1,(m-2)/2}\cup K_1\right)\cong\Smk.
\]
\end{proof}

\begin{proposition}[Reduction after Zheng--Zhang]\label{prop:reduction}
Assume $m\ge 18$ is even. In the branch decomposition of~\cite[Claims~3.1--3.5]{ZhengZhang2025}, every branch other than
\[
G[\Ap]\cong K_4
\qquad\text{with}\qquad
e(\Ap,W)\le 3
\]
is either impossible under the assumption $\rho(G)\ge \rhoext(m)$ or is the equality graph $\Smk$. Consequently, to rule out any \emph{new} counterexample to Theorem~\ref{thm:main}, it suffices to study the $K_4$-branch above.
\end{proposition}

\begin{proof}
By Lemma~\ref{lem:Aplus-structure}, one has $\Ap\neq\varnothing$, the graph $G[\Ap]$ is connected, and $G[\Ap]$ is $(P_2\cup P_3)$-free. Therefore Lemma~\ref{lem:claim33} applies, and $G[\Ap]\cong H_i$ for some $i\in[7]$.

If $i\in\{1,2,3,4,5\}$ and $e(\Ap,W)\le 3$, then Lemma~\ref{lem:nonK4-case1} applies, while if $i\in\{1,2,3,4,5\}$ and $e(\Ap,W)\ge 4$, then Lemma~\ref{lem:nonK4-large-interface} applies. If $i=6$ and $e(\Ap,W)\ge 4$, then Lemma~\ref{lem:K4-large-interface} applies. If $i=7$, then Lemma~\ref{lem:H7-terminal} shows directly that only $\Smk$ can satisfy $\rho(G)\ge \rhoext(m)$.

The only branch not covered by these statements is therefore
\[
G[\Ap]\cong K_4
\qquad\text{with}\qquad
e(\Ap,W)\le 3,
\]
which is exactly the branch treated in the rest of the paper.
\end{proof}

Table~\ref{tab:reduction-map} summarizes the reduction proof map.
\begin{table}[htbp]
\centering
\small
\caption{Reduction map for the Perron-neighborhood branch analysis.}
\label{tab:reduction-map}
\begin{tabular}{@{}c c p{0.38\textwidth}@{}}
\toprule
Core $G[\Ap]$ & Interface condition & Outcome \\
\midrule
$H_1,\dots,H_5$ & $e(\Ap,W)\le3$ & excluded by Lemma~\ref{lem:nonK4-case1} \\
$H_1,\dots,H_5$ & $e(\Ap,W)\ge4$ & excluded by Lemma~\ref{lem:nonK4-large-interface} \\
$K_4$ & $e(\Ap,W)\ge4$ & excluded by Lemma~\ref{lem:K4-large-interface} \\
$H_7=K_{1,t}$ & all & terminal star family; only $\Smk$ survives \\
$K_4$ & $e(\Ap,W)\le3$ & threshold-bearing branch treated below \\
\bottomrule
\end{tabular}
\end{table}

Henceforth we work entirely in the final row of the table.

\section{The interface independence principle}\label{sec:local}

Throughout this section we assume
\[
G[\Ap]\cong K_4
\qquad\text{and}\qquad
e(\Ap,W)\le 3.
\]

We first record the local restrictions used in the $K_4$-branch.

\begin{lemma}\label{lem:local-basic}
For every $w\in W$,
\begin{enumerate}[label={\rm(\arabic*)}]
\item $d_{\Ap}(w)\le 1$,
\item $d_{\Az}(w)\le 1$,
\item one cannot have both $d_{\Ap}(w)\ge 1$ and $d_{\Az}(w)\ge 1$.
\end{enumerate}
In particular, since every $w\in W$ has degree at least $2$ by Lemma~\ref{lem:W-degree}, every vertex of $W$ has a neighbor in $W$.
\end{lemma}

\begin{proof}
Let $w\in W$.

If $w$ had two neighbors $a,b\in \Ap$, then $awb$ would be a triangle. Since $G[\Ap]\cong K_4$, choose two vertices $c,d\in \Ap\setminus\{a,b\}$. Then
\[
a-u^*-c-d-a
\]
is a $4$-cycle sharing exactly the vertex $a$ with the triangle $awb$, a contradiction. Thus $d_{\Ap}(w)\le 1$.

If $w$ had two neighbors $s,t\in \Az$, then
\[
u^*-s-w-t-u^*
\]
would be a $4$-cycle. Since $G[\Ap]\cong K_4$, choose any edge $ab\in E(\Ap)$. Then $u^*ab$ is a triangle sharing exactly the vertex $u^*$ with that $4$-cycle, again a contradiction. Thus $d_{\Az}(w)\le 1$.

Finally, if $w$ had neighbors $a\in \Ap$ and $s\in \Az$, then
\[
a-w-s-u^*-a
\]
would be a $4$-cycle. Choose two vertices $b,c\in \Ap\setminus\{a\}$. Since $G[\Ap]\cong K_4$, the vertices $a,b,c$ form a triangle. This triangle shares exactly the vertex $a$ with the above $4$-cycle, a contradiction.

Hence (1), (2), and (3) all hold. The final statement follows because every vertex of $W$ has degree at least $2$ by Lemma~\ref{lem:W-degree}, but can have at most one neighbor in $\Ap\cup\Az$.
\end{proof}

\begin{proposition}[Interface independence principle]\label{prop:interface-independence}
Let
\[
U:=\{w\in W:d_{\Ap}(w)=1\}.
\]
Then $U$ is an independent set in $G[W]$.
\end{proposition}

\begin{proof}[Proof of Proposition~\ref{prop:interface-independence}]
Suppose to the contrary that $w_1w_2\in E(W)$ with $w_1,w_2\in U$. Let $a_i$ be the unique neighbor of $w_i$ in $\Ap$.

If $a_1\neq a_2$, then
\[
a_1-w_1-w_2-a_2-a_1
\]
is a $4$-cycle. Choose $b\in\Ap\setminus\{a_1,a_2\}$. The triangle $u^*a_1bu^*$ shares exactly the vertex $a_1$ with this $4$-cycle, giving a copy of $H(4,3)$, a contradiction.

Thus $a_1=a_2=:a$. Then $a,w_1,w_2$ form a triangle. Since $\Ap\cong K_4$, choose distinct vertices $b,c\in \Ap\setminus\{a\}$. The cycle
\[
a-u^*-b-c-a
\]
is a $4$-cycle sharing exactly the vertex $a$ with the triangle $aw_1w_2a$, which creates a copy of $H(4,3)$, a contradiction.
\end{proof}

\begin{corollary}\label{cor:ewalpha}
In the $K_4$-branch one has
\[
e(\Ap,W)=|U|\le \alpha(G[W]).
\]
\end{corollary}

\begin{proof}
Each vertex of $U$ has exactly one neighbor in $\Ap$, and no vertex of $W\setminus U$ has any neighbor in $\Ap$. Thus every edge between $\Ap$ and $W$ arises from exactly one vertex of $U$, so $e(\Ap,W)=|U|$. Since $U$ is an independent set by Proposition~\ref{prop:interface-independence}, we obtain $|U|\le \alpha(G[W])$.
\end{proof}

This corollary is the point at which the threshold problem becomes finite. In the $K_4$-branch, \eqref{eq:zz3} gives $e(W)\le3$. The interface independence principle then bounds $e(\Ap,W)$ by the independence number of a graph with at most three edges and no isolated vertices.

\section{The obstruction family and the case $e(W)=0$}\label{sec:ew0}

When $e(W)=0$ in the $K_4$-branch, Lemma~\ref{lem:local-basic} forces $W=\varnothing$. Thus all vertices outside $\Ap\cup\{u^*\}$ lie in $\Az$ and are leaves adjacent only to $u^*$, so the graph is forced.

\begin{proposition}\label{prop:ew0}
Assume $G[\Ap]\cong K_4$ and $e(W)=0$. Then
\[
G\cong T_m:=K_1\vee\bigl(K_4\cup (m-10)K_1\bigr).
\]
Moreover:
\begin{enumerate}[label={\rm(\arabic*)}]
\item $T_m$ is $H(4,3)$-free;
\item $\rho(T_m)>\rhoext(m)$ for $m\in\{10,12,14,16\}$;
\item $\rho(T_m)<\rhoext(m)$ for every even $m\ge 18$.
\end{enumerate}
\end{proposition}

\begin{proof}
By Lemma~\ref{lem:local-basic}, every vertex of $W$ has a neighbor in $W$. Since $e(W)=0$, this forces $W=\varnothing$. Thus every vertex outside $\Ap\cup\{u^*\}$ lies in $\Az$ and is adjacent only to $u^*$, so the graph is exactly $T_m$.

The partition
\[
\{u^*\},\qquad \Ap,\qquad \Az
\]
is equitable, with quotient matrix
\[
Q_T=
\begin{pmatrix}
0&4&m-10\\
1&3&0\\
1&0&0
\end{pmatrix}.
\]
Hence the characteristic polynomial is
\[
q_T(x)=x^3-3x^2+(6-m)x+3m-30.
\]

To see that $T_m$ is $H(4,3)$-free, note that every triangle of $T_m$ lies entirely in $\{u^*\}\cup \Ap$, while every $4$-cycle also lies in $\{u^*\}\cup \Ap$ because vertices of $\Az$ are leaves. Any triangle and any $4$-cycle therefore share at least two vertices, never exactly one.

For the remaining small even values, we avoid decimal comparison by inserting a rational number strictly between the two roots. Let
\[
r_{10}=\frac72,\qquad r_{12}=4,
\qquad r_{14}=\frac{21}{5},
\qquad r_{16}=\frac{869}{200}.
\]
A direct calculation gives
\[
\begin{array}{c|c|c}
 m & p_m(r_m) & q_T(r_m)\\ \hline
 10 & 57/16 & -63/8\\
 12 & 29 & -2\\
 14 & 12381/625 & -54/125\\
 16 & 837095921/1600000000 & -461691/8000000
\end{array}
\]
and in each of these four cases both $p_m$ and $q_T$ are increasing on $[r_m,\infty)$. Hence the largest root of $p_m$ is less than $r_m$, while the Perron root of $T_m$, which is the largest root of $q_T$, is greater than $r_m$. Therefore $\rho(T_m)>\rhoext(m)$ for $m\in\{10,12,14,16\}$.

For $m\ge 18$, Appendix~\ref{app:signchecks} proves that $q_T(\Lm)>0$ and that $q_T$ is strictly increasing on $[\Lm,\infty)$. Since $\rhoext(m)>\Lm$, it follows that $q_T(\rhoext(m))>0$, so the Perron root of $T_m$ lies strictly below $\rhoext(m)$.
\end{proof}

\begin{remark}
Proposition~\ref{prop:ew0} shows that the exact threshold, whatever it is, cannot be smaller than $18$.
\end{remark}

\section{The case $e(W)=1$: exhaustive classification and exact comparison}\label{sec:ew1}

We now classify the admissible graphs in the case $e(W)=1$ and compare each family with $\rhoext(m)$.

\begin{proposition}\label{prop:ew1-class}
Assume $G[\Ap]\cong K_4$ and $e(W)=1$. Then every admissible $H(4,3)$-free graph is isomorphic to exactly one of the following three families:
\begin{enumerate}[label={\rm(\arabic*)}]
\item the \emph{same-$\Az$ family}: the two vertices of the unique edge of $W$ are both adjacent to the same special vertex of $\Az$;
\item the \emph{distinct-$\Az$ family}: the two vertices of the unique edge of $W$ are adjacent to two distinct special vertices of $\Az$;
\item the \emph{mixed family}: one endpoint of the unique edge of $W$ is adjacent to one vertex of $\Ap$, while the other endpoint is adjacent to one special vertex of $\Az$.
\end{enumerate}
All remaining vertices of $\Az$ are leaves adjacent only to $u^*$.
\end{proposition}

\begin{proof}
Since $e(W)=1$, the graph $G[W]$ is a single edge $xy$. By Proposition~\ref{prop:interface-independence}, at most one of $x,y$ may belong to $U$, hence at most one of them may be adjacent to a vertex of $\Ap$.

By Lemma~\ref{lem:local-basic}, each endpoint has at most one neighbor in $\Az$, and no endpoint can meet both $\Ap$ and $\Az$. Since every endpoint has degree at least $2$, the possibilities are exhausted as follows.

If neither endpoint meets $\Ap$, then both must meet $\Az$. Their $\Az$-neighbors may coincide or may be distinct, giving families \rm(1) and \rm(2).

If exactly one endpoint meets $\Ap$, then the other endpoint cannot also meet $\Ap$ and must instead meet one vertex of $\Az$, giving family \rm(3).

No other possibility survives the degree and compatibility constraints.
\end{proof}

We now treat the three families separately.

\begin{proposition}\label{prop:sameA0}
In the same-$\Az$ family, the characteristic polynomial of the equitable quotient matrix is
\[
f_{\mathrm{same}}(x)=x^5-4x^4+(10-m)x^3+(4m-42)x^2+(19-m)x+84-6m.
\]
For every even $m\ge 18$, the Perron root of this family is strictly smaller than $\rhoext(m)$.
\end{proposition}

\begin{proof}
The equitable partition
\[
\{u^*\},\qquad \Ap,\qquad \{z\},\qquad \{x,y\},\qquad L
\]
has quotient matrix
\[
Q_{\mathrm{same}}=
\begin{pmatrix}
0&4&1&0&m-14\\
1&3&0&0&0\\
1&0&0&2&0\\
0&0&1&1&0\\
1&0&0&0&0
\end{pmatrix},
\]
where $L$ denotes the ordinary leaves in $\Az$. Its characteristic polynomial is exactly the displayed polynomial.

A direct division gives
\[
f_{\mathrm{same}}(x)=(x-4)p_m(x)+R_{\mathrm{same}}(x),
\]
where
\[
R_{\mathrm{same}}(x)=10x^3+(m-44)x^2+\Bigl(28-\frac{11m}{2}\Bigr)x+80-4m.
\]
By Appendix~\ref{app:signchecks}, one has $R_{\mathrm{same}}(x)>0$ for all $x\ge \Lm$ whenever $m\ge 18$. Also $\Lm>4$ for every even $m\ge 18$, so $x-4\ge 0$ on $[\Lm,\infty)$. Since $\rhoext(m)>\Lm$, Lemma~\ref{lem:root-template} applies.
\end{proof}

\begin{proposition}\label{prop:distinctA0}
In the distinct-$\Az$ family, the characteristic polynomial of the equitable quotient matrix is
\[
f_{\mathrm{dist}}(x)=x^5-4x^4+(11-m)x^3+(4m-45)x^2+(28-2m)x+45-3m.
\]
For every even $m\ge 18$, the Perron root of this family is strictly smaller than $\rhoext(m)$.
\end{proposition}

\begin{proof}
The quotient matrix is
\[
Q_{\mathrm{dist}}=
\begin{pmatrix}
0&4&2&0&m-15\\
1&3&0&0&0\\
1&0&0&1&0\\
0&0&1&1&0\\
1&0&0&0&0
\end{pmatrix}.
\]
Dividing $f_{\mathrm{dist}}(x)$ by $p_m(x)$ yields
\[
f_{\mathrm{dist}}(x)=(x-4)p_m(x)+R_{\mathrm{dist}}(x),
\]
where
\[
R_{\mathrm{dist}}(x)=11x^3+(m-47)x^2+\Bigl(37-\frac{13m}{2}\Bigr)x+41-m.
\]
The appendix proves that $R_{\mathrm{dist}}(x)>0$ for all $x\ge \Lm$ and all even $m\ge 18$. Also $\Lm>4$ for every even $m\ge 18$, so $x-4\ge 0$ on $[\Lm,\infty)$. The conclusion again follows from Lemma~\ref{lem:root-template}.
\end{proof}

\begin{proposition}\label{prop:mixed}
In the mixed family, the characteristic polynomial of the equitable quotient matrix is
\begin{align*}
f_{\mathrm{mix}}(x)
&=x^7-2x^6+(3-m)x^5+(2m-24)x^4 \\
&\quad +(6m-61)x^3+(84-6m)x^2+(90-7m)x+2m-28.
\end{align*}
For every even $m\ge 18$, the Perron root of this family is strictly smaller than $\rhoext(m)$.
\end{proposition}

\begin{proof}
The corresponding quotient matrix is
\[
Q_{\mathrm{mix}}=
\begin{pmatrix}
0&1&3&1&0&0&m-14\\
1&0&3&0&1&0&0\\
1&1&2&0&0&0&0\\
1&0&0&0&0&1&0\\
0&1&0&0&0&1&0\\
0&0&0&1&1&0&0\\
1&0&0&0&0&0&0
\end{pmatrix}.
\]
One computes
\[
f_{\mathrm{mix}}(x)=\bigl(x^3-2x^2+3x+m-26\bigr)p_m(x)+R_{\mathrm{mix}}(x),
\]
where
\begin{align*}
R_{\mathrm{mix}}(x)
&=\Bigl(\frac{13m}{2}-56\Bigr)x^3+(m^2-28m+76)x^2 \\
&\quad +(m^2-\frac{73m}{2}+145)x-\frac{m^2}{2}+16m-54.
\end{align*}
The required sign and monotonicity verifications are carried out in Appendix~\ref{app:signchecks}. In particular, the appendix proves that $x^3-2x^2+3x+m-26\ge 0$ and $R_{\mathrm{mix}}(x)>0$ for all $x\ge \Lm$ when $m\ge 18$, so Lemma~\ref{lem:root-template} applies.
\end{proof}

\begin{corollary}\label{cor:ew1done}
If $G[\Ap]\cong K_4$ and $e(W)=1$, then $\rho(G)<\rhoext(m)$ for every even $m\ge 18$.
\end{corollary}

\section{The bottleneck: $e(W)=2$}\label{sec:ew2}

Set
\[
\tau=e(\Ap,W).
\]
The case $e(W)=2$ is the bottleneck in the sharp threshold proof. By the interface independence principle, $\tau\le2$, so we split into $\tau=2,1,0$. The first subcase is uniform in $m$, while the latter two reduce to finite endpoint comparisons, as summarized in Table~\ref{tab:ew2-methods}.
\begin{table}[htbp]
\centering
\small
\caption{Methods used in the $e(W)=2$ bottleneck branch according to $\tau=e(\Ap,W)$.}
\label{tab:ew2-methods}
\begin{tabular}{@{}c c p{0.43\textwidth}@{}}
\toprule
$\tau=e(\Ap,W)$ & remaining values & method \\
\midrule
$2$ & all even $m\ge18$ & structural classification and quotient comparison \\
$1$ & $m=18,20$ & finite shifted-polynomial check \\
$0$ & $m=18$ & finite shifted-polynomial check \\
\bottomrule
\end{tabular}
\end{table}
We begin with the crude inequality that explains where the old threshold came from.

\begin{proposition}\label{prop:ew2}
Assume $G[\Ap]\cong K_4$ and $e(W)=2$. Then
\[
\rho(G)^2-\frac32\rho(G)<18.
\]
Consequently $\rho(G)<\rhoext(m)$ for every even $m\ge 24$.
\end{proposition}

\begin{proof}
A graph with two edges and no isolated vertices is either $P_3$ or $2K_2$, so in either case
\[
\alpha(G[W])=2.
\]
By Corollary~\ref{cor:ewalpha}, this implies $e(\Ap,W)\le 2$.

Since $G[\Ap]\cong K_4$, we have $e(\Ap)=6$ and $|\Ap|=4$. Substituting into \eqref{eq:zz7} gives
\[
\rho(G)^2-\rho(G)\Bigl(\frac72-e(W)\Bigr)<16+e(\Ap,W).
\]
Now $e(W)=2$ and $e(\Ap,W)\le 2$, so
\[
\rho(G)^2-\frac32\rho(G)<18.
\]
On the other hand, by \eqref{eq:lower32},
\[
\rho(G)^2-\frac32\rho(G)>m-\frac14\sqrt{4m-5}-\frac74.
\]
For even $m\ge 24$, the right-hand side exceeds $18$. Hence the displayed upper bound is impossible.
\end{proof}

We now settle the extremal subcase $\tau=2$ uniformly for all even $m\ge 18$.

\begin{proposition}\label{prop:ew2-t2-p3}
Assume $G[\Ap]\cong K_4$, $e(W)=2$, $G[W]\cong P_3$, and $e(\Ap,W)=2$. Then
\[
\rho(G)<\rhoext(m)
\]
for every even $m\ge 18$.
\end{proposition}

\begin{proof}
Write
\[
G[W]=x-y-z.
\]
Let
\[
U=\{w\in W:d_{\Ap}(w)=1\}.
\]
Since $e(\Ap,W)=2$, one has $|U|=2$. By Proposition~\ref{prop:interface-independence}, $U$ is independent in $G[W]$, so
\[
U=\{x,z\}.
\]
Thus $x$ and $z$ each have a unique neighbor in $\Ap$, while $y$ has no neighbor in $\Ap$.

Let the unique neighbors of $x$ and $z$ in $\Ap$ be $a$ and $b$. Then $a\ne b$: if $a=b$, then
\[
a-x-y-z-a
\]
is a $4$-cycle, while $u^*ac u^*$ is a triangle for any $c\in\Ap\setminus\{a\}$, giving a copy of $H(4,3)$.

By Lemma~\ref{lem:local-basic}, neither $x$ nor $z$ has a neighbor in $\Az$. The middle vertex $y$ already has degree $2$ inside $W$, so by Lemma~\ref{lem:local-basic} it has either no neighbor in $\Az$ or exactly one. Hence there are exactly two families.

In the first family, denoted $\mathcal P_m^{(0)}$, the vertex $y$ has no neighbor in $\Az$. Then all vertices of $\Az$ are ordinary leaves adjacent only to $u^*$, and the edge count gives
\[
|\Az|=m-14.
\]
With the equitable partition
\[
\{u^*\},\quad A_1=\{a,b\},\quad A_2=\Ap\setminus A_1,\quad \{x,z\},\quad \{y\},\quad \Az,
\]
the quotient matrix has characteristic polynomial
\begin{align*}
 f_0(x)&=x^6-2x^5+(4-m)x^4+(2m-27)x^3 \\
&\quad +(6m-68)x^2+(78-5m)x+84-6m.
\end{align*}

In the second family, denoted $\mathcal P_m^{(1)}$, the vertex $y$ has one neighbor $s\in\Az$. The remaining $m-16$ vertices of $\Az$ are ordinary leaves adjacent only to $u^*$, and the characteristic polynomial of the natural equitable quotient is
\begin{align*}
 f_1(x)&=x^7-2x^6+(4-m)x^5+(2m-27)x^4+(7m-82)x^3 \\
&\quad +(114-7m)x^2+(148-10m)x+m-16.
\end{align*}

In both cases the displayed polynomial is the characteristic polynomial of an irreducible equitable quotient, so its Perron root is the spectral radius of the corresponding graph. The sign verification in Appendix~\ref{app:ew2-t2-p3} shows that $f_i(x)>0$ for every $x\ge \Lm$ and $i\in\{0,1\}$. Thus the Perron root of each of these two quotient matrices is less than $\Lm$, and hence less than $\rhoext(m)$ by Lemma~\ref{lem:Lm-lower}. This proves the proposition.
\end{proof}

\begin{proposition}\label{prop:ew2-t2-2k2}
Assume $G[\Ap]\cong K_4$, $e(W)=2$, $G[W]\cong 2K_2$, and $e(\Ap,W)=2$. Then
\[
\rho(G)<\rhoext(m)
\]
for every even $m\ge 18$.
\end{proposition}

\begin{proof}
Write
\[
G[W]=x_1y_1\cup x_2y_2.
\]
Since $e(\Ap,W)=2$, the set $U=\{w\in W:d_{\Ap}(w)=1\}$ has size $2$. By Proposition~\ref{prop:interface-independence}, it is independent in $G[W]$, so after relabeling
\[
U=\{x_1,x_2\}.
\]
Thus $x_1,x_2$ are the two vertices meeting $\Ap$, while $y_1,y_2$ have no neighbors in $\Ap$.

Each $y_i$ has degree $1$ inside $W$. By Lemma~\ref{lem:local-basic} and Lemma~\ref{lem:W-degree}, each $y_i$ has exactly one neighbor in $\Az$. Hence
\[
e(\Az,W)=2.
\]
The edge count gives
\[
m=(4+|\Az|)+6+(2+2)+2,
\]
so
\[
|\Az|=m-16.
\]

There are two independent binary choices. The vertices $x_1,x_2$ either meet the same vertex of $\Ap$ or distinct vertices of $\Ap$; similarly, $y_1,y_2$ either meet the same vertex of $\Az$ or distinct vertices of $\Az$. Thus every graph in this subcase belongs to one of four families, denoted
\[
\mathcal F_m^{ss},\quad \mathcal F_m^{sd},\quad \mathcal F_m^{ds},\quad \mathcal F_m^{dd},
\]
where the first superscript records whether the two $\Ap$-neighbors are the same or distinct, and the second records whether the two $\Az$-neighbors are the same or distinct.

The natural equitable quotient matrices for these four families have characteristic polynomials
\begin{align*}
 f_{ss}(x)&=x^7-2x^6+(4-m)x^5+(2m-26)x^4+(8m-99)x^3 \\
&\quad +(164-10m)x^2+(202-13m)x+8m-136,\\[2mm]
 f_{sd}(x)&=x^7-2x^6+(5-m)x^5+(2m-28)x^4+(7m-92)x^3 \\
&\quad +(132-8m)x^2+(128-8m)x+4m-72,\\[2mm]
 f_{ds}(x)&=x^7-2x^6+(5-m)x^5+(2m-29)x^4+(7m-89)x^3 \\
&\quad +(122-7m)x^2+(176-11m)x+2m-34,\\[2mm]
 f_{dd}(x)&=x^7-2x^6+(6-m)x^5+(2m-31)x^4+(6m-81)x^3 \\
&\quad +(87-5m)x^2+(114-7m)x+(m-18).
\end{align*}
In each of the four cases, the displayed polynomial is the characteristic polynomial of an irreducible equitable quotient, so its Perron root is the spectral radius of the corresponding graph. The decomposition and positivity checks in Appendix~\ref{app:ew2-t2-2k2} show that each of these four Perron roots is strictly less than $\rhoext(m)$ for every even $m\ge 18$. This completes the proof.
\end{proof}

\begin{corollary}\label{cor:ew2-t2done}
Assume $G[\Ap]\cong K_4$, $e(W)=2$, and $e(\Ap,W)=2$. Then
\[
\rho(G)<\rhoext(m)
\]
for every even $m\ge 18$.
\end{corollary}

\begin{proof}
Since every vertex of $W$ has a neighbor in $W$ and $e(W)=2$, one has $G[W]\cong P_3$ or $G[W]\cong 2K_2$. The conclusion follows from Propositions~\ref{prop:ew2-t2-p3} and~\ref{prop:ew2-t2-2k2}.
\end{proof}

We next settle the subcase $\tau=1$. The crude inequality already disposes of this subcase for all even $m\neq 18,20$, so only two finite values require a separate comparison.
\begin{proposition}\label{prop:ew2-t1}
Assume $G[\Ap]\cong K_4$, $e(W)=2$, and $e(\Ap,W)=1$. Then
\[
\rho(G)<\rhoext(m)
\]
for every even $m\ge 18$.
\end{proposition}

\begin{proof}
Substituting $e(\Ap)=6$, $|\Ap|=4$, $e(W)=2$, and $e(\Ap,W)=1$ into \eqref{eq:zz7} gives
\[
\rho(G)^2-\frac32\rho(G)<17.
\]
For every even $m\ge 22$, this contradicts \eqref{eq:lower32}, since
\[
m-\frac14\sqrt{4m-5}-\frac74>17.
\]
Thus it remains to handle $m=18$ and $m=20$.

We first classify the possible graphs. Since every vertex of $W$ has a neighbor in $W$ and $e(W)=2$, one has
\[
G[W]\cong P_3\quad\text{or}\quad G[W]\cong 2K_2.
\]
Let $U=\{w\in W:d_{\Ap}(w)=1\}$. Since $e(\Ap,W)=1$, the set $U$ has size $1$.

First suppose $G[W]\cong P_3$, and write
\[
G[W]=x-y-z.
\]
If the middle vertex $y$ is the unique vertex of $W$ meeting $\Ap$, then $x$ and $z$ each have degree $1$ inside $W$ and hence must each meet $\Az$. Their $\Az$-neighbors are either the same or distinct. This gives two families, denoted
\[
\mathcal P^{M,s}_m,\qquad \mathcal P^{M,d}_m.
\]
If instead an endpoint, say $x$, is the unique vertex meeting $\Ap$, then the opposite endpoint $z$ must meet $\Az$, while the middle vertex $y$ may have no neighbor in $\Az$, may share $z$'s neighbor in $\Az$, or may meet a distinct vertex of $\Az$. This gives three more families, denoted
\[
\mathcal P^{E,0}_m,
\qquad
\mathcal P^{E,s}_m,
\qquad
\mathcal P^{E,d}_m.
\]
The characteristic polynomials of the natural irreducible equitable quotients for these five families are
\begin{align*}
f_{M,s}(x)&=x^8-2x^7+(3-m)x^6+(2m-24)x^5+(8m-89)x^4 \\
&\quad +(162-10m)x^3+(208-14m)x^2+(4m-64)x,\\[1mm]
f_{M,d}(x)&=x^7-2x^6+(4-m)x^5+(2m-26)x^4+(7m-82)x^3 \\
&\quad +(130-8m)x^2+(151-10m)x+(2m-34),\\[1mm]
f_{E,0}(x)&=x^8-2x^7+(3-m)x^6+(2m-24)x^5+(7m-72)x^4 \\
&\quad +(120-8m)x^3+(158-12m)x^2+(6m-94)x+(3m-45),\\[1mm]
f_{E,s}(x)&=x^8-2x^7+(3-m)x^6+(2m-26)x^5+(8m-83)x^4 \\
&\quad +(144-8m)x^3+(297-20m)x^2-10x+(7m-112),\\[1mm]
f_{E,d}(x)&=x^9-2x^8+(3-m)x^7+(2m-24)x^6+(8m-86)x^5 \\
&\quad +(156-10m)x^4+(242-17m)x^3+(10m-158)x^2 \\
&\quad +(7m-115)x+(34-2m).
\end{align*}

Now suppose $G[W]\cong 2K_2$, and write
\[
G[W]=x_1y_1\cup x_2y_2.
\]
After relabeling, assume $x_1$ is the unique vertex of $W$ meeting $\Ap$. Then $y_1,x_2,y_2$ each have degree $1$ inside $W$ and no neighbor in $\Ap$, so each must meet $\Az$. Up to isomorphism, there are four possibilities according to the partition of $\{y_1,x_2,y_2\}$ determined by their $\Az$-neighbors: all three meet the same vertex, a nonedge pair shares a vertex, the adjacent pair $x_2,y_2$ shares a vertex, or all three meet distinct vertices. We denote the corresponding families by
\[
\mathcal K^{[3]}_m,
\qquad
\mathcal K^{mix}_m,
\qquad
\mathcal K^{edge}_m,
\qquad
\mathcal K^{[1,1,1]}_m.
\]
The characteristic polynomials of the corresponding natural irreducible equitable quotients are
\begin{align*}
f_{[3]}(x)&=x^8-3x^7+(6-m)x^6+(3m-32)x^5+(6m-65)x^4 \\
&\quad +(249-16m)x^3+(159-11m)x^2+(17m-289)x+(4m-68),\\[1mm]
f_{mix}(x)&=x^{10}-2x^9+(3-m)x^8+(2m-24)x^7+(9m-99)x^6 \\
&\quad +(194-12m)x^5+(363-25m)x^4+(20m-342)x^3 \\
&\quad +(22m-370)x^2+(138-8m)x+(54-3m),\\[1mm]
f_{edge}(x)&=x^9-3x^8+(6-m)x^7+(3m-32)x^6+(6m-63)x^5 \\
&\quad +(245-16m)x^4+(179-13m)x^3+(21m-361)x^2 \\
&\quad +(12m-202)x+(72-4m),\\[1mm]
f_{[1,1,1]}(x)&=x^{11}-2x^{10}+(3-m)x^9+(2m-24)x^8+(9m-98)x^7 \\
&\quad +(192-12m)x^6+(372-26m)x^5+(22m-374)x^4 \\
&\quad +(27m-450)x^3+(206-12m)x^2+(129-7m)x+(2m-38).
\end{align*}
The family $\mathcal K^{[1,1,1]}_m$ exists only when there are at least three vertices in $\Az$, so in the present range it is relevant only for $m=20$.

At the two remaining values, we compare to simple lower bounds for $\rhoext(m)$. One computes
\[
p_{18}\!\left(\frac92\right)=-\frac{295}{16}<0,
\qquad
p_{20}\!\left(\frac{24}{5}\right)=-\frac{4599}{625}<0.
\]
Hence
\[
\rhoext(18)>\frac92,
\qquad
\rhoext(20)>\frac{24}{5}.
\]
By the shifted-coefficient check in Appendix~\ref{app:ew2-t1}, each of the displayed family polynomials is positive on $[9/2,\infty)$ when $m=18$ and on $[24/5,\infty)$ when $m=20$, whenever the corresponding family exists. Therefore the Perron root of every family in the $\tau=1$ branch lies strictly below $\rhoext(m)$ for $m=18,20$.
\end{proof}

\begin{corollary}\label{cor:ew2-t1done}
Assume $G[\Ap]\cong K_4$, $e(W)=2$, and $e(\Ap,W)=1$. Then $\rho(G)<\rhoext(m)$ for every even $m\ge 18$.
\end{corollary}

We finally settle the last subcase, namely $\tau=0$.

\begin{proposition}\label{prop:ew2-t0}
Assume $G[\Ap]\cong K_4$, $e(W)=2$, and $e(\Ap,W)=0$. Then
\[
\rho(G)<\rhoext(m)
\]
for every even $m\ge 18$.
\end{proposition}

\begin{proof}
Substituting $e(\Ap)=6$, $|\Ap|=4$, $e(W)=2$, and $e(\Ap,W)=0$ into \eqref{eq:zz7} gives
\[
\rho(G)^2-\frac32\rho(G)<16.
\]
For every even $m\ge 20$, this contradicts \eqref{eq:lower32}, since
\[
m-\frac14\sqrt{4m-5}-\frac74>16.
\]
Thus it remains to handle $m=18$.

Assume now that $m=18$. Since $e(W)=2$ and every vertex of $W$ has a neighbor in $W$, one has
\[
G[W]\cong P_3\quad\text{or}\quad G[W]\cong 2K_2.
\]
Because $e(\Ap,W)=0$, no vertex of $W$ meets $\Ap$. By Lemma~\ref{lem:local-basic}, each vertex of $W$ has at most one neighbor in $\Az$. The edge count gives
\[
18=(4+|\Az|)+6+e(\Az,W)+2,
\]
so
\begin{equation}\label{eq:ew2-t0-count}
|\Az|+e(\Az,W)=6.
\end{equation}

First suppose $G[W]\cong P_3$, say $G[W]=x-y-z$. The endpoints $x$ and $z$ each have degree $1$ inside $W$, so each must meet $\Az$. The middle vertex $y$ may or may not meet $\Az$.

If $y$ has no neighbor in $\Az$, then $e(\Az,W)=2$ and $|\Az|=4$. The two endpoints either meet the same vertex of $\Az$ or distinct vertices of $\Az$. We denote the corresponding families by
\[
\mathcal P^{0,s},\qquad \mathcal P^{0,d}.
\]
If $y$ has a neighbor in $\Az$, then $e(\Az,W)=3$ and $|\Az|=3$. Up to the symmetry $x\leftrightarrow z$, there are four possibilities for the partition of $\{x,y,z\}$ induced by their $\Az$-neighbors: all three share one vertex, an adjacent pair shares one vertex, the nonadjacent pair $x,z$ shares one vertex, or all three meet distinct vertices. We denote these families by
\[
\mathcal P^{1,[3]},\quad
\mathcal P^{1,adj},\quad
\mathcal P^{1,nonadj},\quad
\mathcal P^{1,[1,1,1]}.
\]

Now suppose $G[W]\cong 2K_2$, say $G[W]=x_1y_1\cup x_2y_2$. Every vertex of $W$ has degree $1$ inside $W$, so every vertex of $W$ must meet $\Az$. Thus $e(\Az,W)=4$, and \eqref{eq:ew2-t0-count} gives $|\Az|=2$. Up to isomorphism, there are four possibilities:
\[
\mathcal K^{[4]},\quad
\mathcal K^{[3,1]},\quad
\mathcal K^{[2,2]_e},\quad
\mathcal K^{[2,2]_c}.
\]
Here $\mathcal K^{[4]}$ means that all four vertices of $W$ meet the same vertex of $\Az$; $\mathcal K^{[3,1]}$ means that three meet one vertex and the fourth meets the other; $\mathcal K^{[2,2]_e}$ means that the endpoints of each edge of $G[W]$ meet the same vertex of $\Az$; and $\mathcal K^{[2,2]_c}$ means that one endpoint from each edge meets one vertex of $\Az$ and the other endpoints meet the other.

Thus every graph in the $m=18$, $\tau=0$ branch belongs to one of the ten explicit families above. Let $f_F(x)$ be the characteristic polynomial of the natural irreducible equitable quotient for a family $F$. We have
\[
p_{18}\!\left(\frac{17}{4}\right)=-\frac{15071}{256}<0,
\]
so
\[
\rhoext(18)>\frac{17}{4}.
\]
The shifted-polynomial check in Appendix~\ref{app:ew2-t0} shows that, for each of the ten families $F$,
\[
f_F\!\left(t+\frac{17}{4}\right)>0
\qquad (t\ge 0).
\]
Therefore each quotient polynomial has no root in $[17/4,\infty)$, and the Perron root of every family is strictly less than $17/4<\rhoext(18)$. This proves the proposition.
\end{proof}

\begin{corollary}\label{cor:ew2done}
If $G[\Ap]\cong K_4$ and $e(W)=2$, then $\rho(G)<\rhoext(m)$ for every even $m\ge 18$.
\end{corollary}

\begin{proof}
By Corollary~\ref{cor:ewalpha}, one has $e(\Ap,W)\le 2$. The cases $e(\Ap,W)=2,1,0$ are handled by Corollaries~\ref{cor:ew2-t2done} and~\ref{cor:ew2-t1done}, and Proposition~\ref{prop:ew2-t0}, respectively.
\end{proof}
\section{The endpoint case $e(W)=3$}\label{sec:ew3}

For $e(W)=3$, the inequalities rule out all even $m\ge20$. Thus only the endpoint $m=18$ remains, where the interface independence principle forces a short finite classification.

\begin{proposition}\label{prop:ew3-large}
Assume $G[\Ap]\cong K_4$ and $e(W)=3$. Then
\[
\rho(G)^2-\frac12\rho(G)<19.
\]
Consequently $\rho(G)<\rhoext(m)$ for every even $m\ge 20$.
\end{proposition}

\begin{proof}
By the standing hypothesis of the $K_4$-branch, one has $e(\Ap,W)\le 3$. Since $e(\Ap)=6$, $|\Ap|=4$, and $e(W)=3$, Eq.~\eqref{eq:zz7} gives
\[
\rho(G)^2-\frac12\rho(G)<16+e(\Ap,W)\le 19.
\]
The conclusion follows because \eqref{eq:lower12} shows
\[
\rho(G)^2-\frac12\rho(G)>m+\frac14\sqrt{4m-5}-\frac54>19
\]
for every even $m\ge 20$.
\end{proof}

It remains to exclude the endpoint value $m=18$.

\begin{proposition}\label{prop:ew3-m18}
Assume $m=18$, $G[\Ap]\cong K_4$, and $e(W)=3$. Then
\[
\rho(G)<\rhoext(18).
\]
\end{proposition}

\begin{proof}
Let
\[
U=\{w\in W:d_{\Ap}(w)=1\}.
\]
By Corollary~\ref{cor:ewalpha},
\[
e(\Ap,W)=|U|\le \alpha(G[W]).
\]
If $e(\Ap,W)\le 2$, then \eqref{eq:zz7}, with $e(\Ap)=6$, $|\Ap|=4$, and $e(W)=3$, gives
\[
\rho(G)^2-\frac12\rho(G)<16+e(\Ap,W)\le 18.
\]
This contradicts \eqref{eq:lower12} at $m=18$, since
\[
18+\frac{\sqrt{67}}4-\frac54>18.
\]
Hence
\[
e(\Ap,W)=3.
\]

Since every vertex of $W$ has a neighbor in $W$ by Lemma~\ref{lem:local-basic}, the graph $G[W]$ has three edges and no isolated vertices. Thus
\[
G[W]\cong K_3,\quad P_4,\quad K_{1,3},\quad P_3\cup K_2,\quad \text{or}\quad 3K_2.
\]
The first two possibilities are impossible because $\alpha(K_3)=1$ and $\alpha(P_4)=2$, while $|U|=3$. Hence
\[
G[W]\in\{K_{1,3},\,P_3\cup K_2,\,3K_2\}.
\]
Also, the edge count gives
\[
18=|N|+e(\Ap)+e(N,W)+e(W)
=(4+|\Az|)+6+\bigl(3+e(\Az,W)\bigr)+3,
\]
so
\begin{equation}\label{eq:ew3-a0count}
|\Az|+e(\Az,W)=2.
\end{equation}

We now classify the remaining possibilities.

First suppose $G[W]\cong 3K_2$. Since $U$ is an independent set of size $3$, it contains one endpoint from each edge of $G[W]$. The other three endpoints have no neighbor in $\Ap$ and have degree $1$ inside $W$. Each must therefore have a neighbor in $\Az$, so $e(\Az,W)\ge 3$, contradicting \eqref{eq:ew3-a0count}. Thus $G[W]\not\cong 3K_2$.

Next suppose $G[W]\cong K_{1,3}$, with center $x$ and leaves $y_1,y_2,y_3$. Since the leaves are the unique independent set of size $3$, we have
\[
U=\{y_1,y_2,y_3\}.
\]
The three leaves must meet three distinct vertices of $\Ap$. Indeed, if two leaves $y_i,y_j$ shared the same neighbor $a\in\Ap$, then
\[
a-y_i-x-y_j-a
\]
would be a $4$-cycle, while $u^*ab u^*$ is a triangle for any $b\in\Ap\setminus\{a\}$, giving a fish.

The leaves cannot meet $\Az$ by Lemma~\ref{lem:local-basic}. Hence \eqref{eq:ew3-a0count} leaves exactly two possibilities:
\begin{enumerate}[label={\rm(\roman*)}]
\item $|\Az|=2$ and $e(\Az,W)=0$;
\item $|\Az|=1$ and $e(\Az,W)=1$, in which case the unique edge from $\Az$ to $W$ joins the unique vertex of $\Az$ to the center $x$.
\end{enumerate}
Thus the star case gives two graphs, denoted $\mathcal S_2$ and $\mathcal S_1$, respectively.

Finally suppose $G[W]\cong P_3\cup K_2$. Write the path as $p_1p_2p_3$ and the separate edge as $k_1k_2$. Since $U$ is an independent set of size $3$, after relabeling we may assume
\[
U=\{p_1,p_3,k_1\}.
\]
The vertex $k_2$ has degree $1$ inside $W$ and no neighbor in $\Ap$, so it must have a neighbor in $\Az$. By \eqref{eq:ew3-a0count}, this forces
\[
|\Az|=1,\qquad e(\Az,W)=1,
\]
and the unique edge from $\Az$ to $W$ is incident with $k_2$.

The vertices $p_1$ and $p_3$ cannot share their neighbor in $\Ap$: if both met $a\in\Ap$, then
\[
a-p_1-p_2-p_3-a
\]
would be a $4$-cycle, while $u^*ab u^*$ is a triangle for any $b\in\Ap\setminus\{a\}$. Therefore the neighbors in $\Ap$ of $p_1$ and $p_3$ are distinct. Up to symmetry, there are exactly two possibilities: either $k_1$ shares its $\Ap$-neighbor with one of $p_1,p_3$, or the three vertices $p_1,p_3,k_1$ meet three distinct vertices of $\Ap$. Denote these two graphs by $\mathcal P_{21}$ and $\mathcal P_{111}$.

So every graph in the present branch is one of
\[
\mathcal S_2,\quad \mathcal S_1,\quad \mathcal P_{21},\quad \mathcal P_{111}.
\]
For reference, Table~\ref{tab:ew3-endpoint} summarizes the four endpoint families.
\begin{table}[htbp]
\centering
\small
\caption{Endpoint families remaining in the $m=18$, $e(W)=3$ branch.}
\label{tab:ew3-endpoint}
\begin{tabular}{@{}c c p{0.45\textwidth}@{}}
\toprule
Family & $G[W]$ & additional structure \\
\midrule
$\mathcal S_2$ & $K_{1,3}$ & $|\Az|=2$ and no edge from $\Az$ to $W$ \\
$\mathcal S_1$ & $K_{1,3}$ & the center of the star meets the unique vertex of $\Az$ \\
$\mathcal P_{21}$ & $P_3\cup K_2$ & two of the three $\Ap$-neighbors coincide \\
$\mathcal P_{111}$ & $P_3\cup K_2$ & the three $\Ap$-neighbors are distinct \\
\bottomrule
\end{tabular}
\end{table}

It remains to compare these four graphs spectrally. Let
\[
p_{18}(x)=x^4-18x^2-16x+8.
\]
At
\[
x_0=\frac{43}{10},
\]
one has
\[
p_{18}(x_0)=-\frac{517399}{10000}<0,
\]
so
\[
\rhoext(18)>x_0.
\]

For the four graphs above, direct characteristic-polynomial computations give
\[
\chi_{\mathcal S_2}(x)
=
x(x^2-3)(x^2+x-1)^2\,q_{\mathcal S_2}(x),
\]
where
\[
q_{\mathcal S_2}(x)=x^4-2x^3-10x^2+6;
\]
\[
\chi_{\mathcal S_1}(x)
=
(x^2+x-1)^3\,q_{\mathcal S_1}(x),
\]
where
\[
q_{\mathcal S_1}(x)=x^4-3x^3-9x^2+12x+16;
\]
\[
\chi_{\mathcal P_{21}}(x)
=
(x+1)^2(x^3-3x+1)\,q_{\mathcal P_{21}}(x),
\]
where
\[
q_{\mathcal P_{21}}(x)=x^6-2x^5-12x^4+5x^3+27x^2-3x-12;
\]
and
\[
\chi_{\mathcal P_{111}}(x)
=
(x+2)(x^2-x-1)(x^2+x-1)^2\,q_{\mathcal P_{111}}(x),
\]
where
\[
q_{\mathcal P_{111}}(x)=x^4-3x^3-7x^2+6x+8.
\]
All auxiliary factors have largest real root less than $2$. Therefore, in each case, the Perron root is the largest real root of the displayed factor $q$.

We now check that each displayed factor is positive on $[x_0,\infty)$. At $x_0=43/10$,
\[
q_{\mathcal S_2}(x_0)=\frac{39661}{10000}>0,\qquad
q_{\mathcal S_1}(x_0)=\frac{45491}{10000}>0,
\]
\[
q_{\mathcal P_{21}}(x_0)=\frac{150497989}{1000000}>0,\qquad
q_{\mathcal P_{111}}(x_0)=\frac{77291}{10000}>0.
\]
Also each $q$ is increasing on $[x_0,\infty)$. For $\mathcal S_2$, this follows from
\[
q'_{\mathcal S_2}(x)=2x(2x^2-3x-10)>0\qquad (x\ge x_0).
\]
For $\mathcal S_1$,
\[
q'_{\mathcal S_1}(x)=4x^3-9x^2-18x+12,\qquad
q''_{\mathcal S_1}(x)=12x^2-18x-18,
\]
and $q''_{\mathcal S_1}(x)>0$ for $x\ge x_0$, while
\[
q'_{\mathcal S_1}(x_0)=\frac{43109}{500}>0.
\]
For $\mathcal P_{21}$,
\[
q'_{\mathcal P_{21}}(x)=6x^5-10x^4-48x^3+15x^2+54x-3,
\]
\[
q''_{\mathcal P_{21}}(x)=30x^4-40x^3-144x^2+30x+54.
\]
For $x\ge x_0$,
\[
q''_{\mathcal P_{21}}(x)
=
x^2(30x^2-40x-144)+30x+54>0,
\]
and
\[
q'_{\mathcal P_{21}}(x_0)=\frac{104595979}{50000}>0.
\]
Finally, for $\mathcal P_{111}$,
\[
q'_{\mathcal P_{111}}(x)=4x^3-9x^2-14x+6,\qquad
q''_{\mathcal P_{111}}(x)=12x^2-18x-14,
\]
and $q''_{\mathcal P_{111}}(x)>0$ for $x\ge x_0$, while
\[
q'_{\mathcal P_{111}}(x_0)=\frac{48709}{500}>0.
\]
Thus each $q$ is positive on $[x_0,\infty)$, so every one of the four Perron roots is less than $x_0$. Since $x_0<\rhoext(18)$, we conclude
\[
\rho(G)<\rhoext(18).
\]
\end{proof}

\begin{corollary}\label{cor:ew3done}
If $G[\Ap]\cong K_4$ and $e(W)=3$, then
\[
\rho(G)<\rhoext(m)
\]
for every even $m\ge 18$.
\end{corollary}

\begin{proof}
For even $m\ge 20$, this is Proposition~\ref{prop:ew3-large}; for $m=18$, it is Proposition~\ref{prop:ew3-m18}.
\end{proof}

\section{Proof of the main theorem}\label{sec:proof-main}

\begin{proof}[Proof of Theorem~\ref{thm:main}]
By Proposition~\ref{prop:reduction}, it suffices to treat the branch
\[
G[\Ap]\cong K_4,\qquad e(\Ap,W)\le 3.
\]

If $e(W)=0$, then Proposition~\ref{prop:ew0} gives $\rho(G)<\rhoext(m)$ for every even $m\ge 18$.

If $e(W)=1$, then Corollary~\ref{cor:ew1done} gives $\rho(G)<\rhoext(m)$ for every even $m\ge 18$.

If $e(W)=2$, then Corollary~\ref{cor:ew2done} gives $\rho(G)<\rhoext(m)$ for every even $m\ge 18$.

If $e(W)=3$, then Corollary~\ref{cor:ew3done} gives $\rho(G)<\rhoext(m)$ for every even $m\ge 18$.

Therefore every branch other than the terminal star branch is impossible for even $m\ge 18$. The terminal branch is handled by Lemma~\ref{lem:H7-terminal}, which gives exactly the comparison graph $\Smk$. This proves the theorem.
\end{proof}

\appendix
\setcounter{table}{0}
\renewcommand{\thetable}{A.\arabic{table}}
\renewcommand{\theHtable}{A.\arabic{table}}
\numberwithin{equation}{section}

\section{Exact polynomial checks for the finite families}\label{app:signchecks}

All computations in this appendix are exact. The finite-family polynomials come from the natural equitable partitions described in the corresponding proofs. In the finite endpoint cases, the displayed tables record a positive denominator and a positive lower bound for the shifted coefficients, which makes the root exclusions reproducible without relying on numerical approximations.

The accompanying ancillary file \texttt{finite\_verification.py} verifies the finite shifted-coefficient checks appearing in this appendix using rational arithmetic. It does not verify the graph classifications or the derivation of the quotient matrices, which are proved in the main text. It verifies the final exact polynomial inequalities used to exclude the finite endpoint families.

In several tables the smallest shifted coefficient equals the denominator $D_{F,m}$. This is expected: the quotient polynomials are monic, so after multiplying by $D_{F,m}$, the leading coefficient is $D_{F,m}$; in the listed rows this leading coefficient is also the smallest coefficient.

When an ordinary-leaf class has size $0$, we delete the corresponding row and column in the quotient matrix. Equivalently, if one keeps the displayed formal partition, the displayed polynomial may acquire an extra harmless factor $x$, which does not affect the Perron-root comparison.

\subsection{The obstruction family $T_m$}

We have
\[
q_T(x)=x^3-3x^2+(6-m)x+3m-30.
\]
Evaluating at $\Lm=(1+\sqrt{4m-5})/2$ gives
\[
q_T(\Lm)=\frac{4m+5\sqrt{4m-5}-103}{4},
\]
which is positive for every even $m\ge 18$.

Moreover,
\[
q_T'(x)=3x^2-6x+6-m.
\]
Hence
\[
q_T'(\Lm)=\frac{4m-3\sqrt{4m-5}}{2}>0
\]
for every $m\ge 1$ since
\[
16m^2>9(4m-5).
\]
Finally,
\[
q_T''(x)=6x-6>0\qquad\text{for all }x\ge \Lm>1.
\]
Thus $q_T'$ is increasing on $[\Lm,\infty)$, so $q_T'(x)>0$ for all $x\ge \Lm$. Hence $q_T$ is strictly increasing on $[\Lm,\infty)$, and since $q_T(\Lm)>0$, one has $q_T(x)>0$ for all $x\ge \Lm$. Because $\rhoext(m)>\Lm$, the Perron root of $T_m$ is strictly smaller than $\rhoext(m)$ for every even $m\ge 18$.

\subsection{The same-$\Az$ family}

Recall
\[
R_{\mathrm{same}}(x)=10x^3+(m-44)x^2+\Bigl(28-\frac{11m}{2}\Bigr)x+80-4m.
\]
At $x=\Lm$ this becomes
\[
R_{\mathrm{same}}(\Lm)=\frac{4m^2+11m\sqrt{4m-5}-147m-42\sqrt{4m-5}+482}{4}.
\]
As a function of $m$ its derivative is
\[
\frac{d}{dm}R_{\mathrm{same}}(\Lm)
=
2m-\frac{147}{4}+\frac{66m-139}{4\sqrt{4m-5}}.
\]
For $m\ge 18$ one has
\[
2m-\frac{147}{4}\ge -\frac34,\qquad
\frac{66m-139}{4\sqrt{4m-5}}
>
\frac{66m-139}{8m}\ge \frac{1049}{144}>\frac34,
\]
so $\frac{d}{dm}R_{\mathrm{same}}(\Lm)>0$ on $[18,\infty)$. Since
\[
R_{\mathrm{same}}(\Lm)\big|_{m=18}=-217+39\sqrt{67}>0,
\]
it follows that $R_{\mathrm{same}}(\Lm)>0$ for every even $m\ge 18$.

Next,
\[
R_{\mathrm{same}}'(x)=30x^2+(2m-88)x+28-\frac{11m}{2},
\]
and
\[
R_{\mathrm{same}}''(x)=60x+2m-88.
\]
For $x\ge \Lm$ and $m\ge 18$,
\[
R_{\mathrm{same}}''(x)\ge 60\Lm+2m-88>0,
\]
so $R_{\mathrm{same}}'$ is increasing on $[\Lm,\infty)$. Therefore it suffices to check $R_{\mathrm{same}}'(\Lm)>0$:
\[
R_{\mathrm{same}}'(\Lm)=\frac{2m\sqrt{4m-5}+51m-58\sqrt{4m-5}-92}{2}.
\]
Its derivative with respect to $m$ is
\[
\frac{d}{dm}R_{\mathrm{same}}'(\Lm)
=
\frac{3\bigl(4m+17\sqrt{4m-5}-42\bigr)}{2\sqrt{4m-5}}>0
\]
for $m\ge 18$, and
\[
R_{\mathrm{same}}'(\Lm)\big|_{m=18}=413-11\sqrt{67}>0.
\]
Hence $R_{\mathrm{same}}'(x)>0$ for all $x\ge \Lm$ and all even $m\ge 18$. By Lemma~\ref{lem:root-template}, the Perron root of the same-$\Az$ family is less than $\rhoext(m)$.

\subsection{The distinct-$\Az$ family}

Here
\[
R_{\mathrm{dist}}(x)=11x^3+(m-47)x^2+\Bigl(37-\frac{13m}{2}\Bigr)x+41-m.
\]
Evaluating at $\Lm$ yields
\[
R_{\mathrm{dist}}(\Lm)=\frac{4m^2+11m\sqrt{4m-5}-143m-31\sqrt{4m-5}+349}{4}.
\]
Its derivative with respect to $m$ is
\[
\frac{d}{dm}R_{\mathrm{dist}}(\Lm)
=
2m-\frac{143}{4}+\frac{66m-117}{4\sqrt{4m-5}}>0
\]
for $m\ge 18$, since the first term is at least $1/4$ and the second is positive. Also
\[
R_{\mathrm{dist}}(\Lm)\big|_{m=18}=-\frac{929}{4}+\frac{167}{4}\sqrt{67}>0.
\]
Hence $R_{\mathrm{dist}}(\Lm)>0$ for every even $m\ge 18$.

Furthermore,
\[
R_{\mathrm{dist}}'(x)=33x^2+(2m-94)x+37-\frac{13m}{2},
\qquad
R_{\mathrm{dist}}''(x)=66x+2m-94.
\]
Since $R_{\mathrm{dist}}''(x)>0$ for all $x\ge \Lm$ and $m\ge 18$, the derivative is increasing on $[\Lm,\infty)$. It therefore suffices to evaluate at $\Lm$:
\[
R_{\mathrm{dist}}'(\Lm)=\frac{2m\sqrt{4m-5}+55m-61\sqrt{4m-5}-86}{2}.
\]
Its derivative with respect to $m$ is
\[
\frac{d}{dm}R_{\mathrm{dist}}'(\Lm)
=
\frac{12m+55\sqrt{4m-5}-132}{2\sqrt{4m-5}}>0
\]
for $m\ge 18$, and
\[
R_{\mathrm{dist}}'(\Lm)\big|_{m=18}=452-\frac{25}{2}\sqrt{67}>0.
\]
Thus $R_{\mathrm{dist}}'(x)>0$ for all $x\ge \Lm$ and all even $m\ge 18$. By Lemma~\ref{lem:root-template}, the distinct-$\Az$ family lies below $\rhoext(m)$.

\subsection{The mixed family}

Here the quotient factor is
\[
q_{\mathrm{mix}}(x)=x^3-2x^2+3x+m-26,
\]
and the remainder polynomial is
\begin{align*}
R_{\mathrm{mix}}(x)
&=\Bigl(\frac{13m}{2}-56\Bigr)x^3+(m^2-28m+76)x^2 \\
&\quad +(m^2-\frac{73m}{2}+145)x-\frac{m^2}{2}+16m-54.
\end{align*}

First, $q_{\mathrm{mix}}$ is increasing on $[0,\infty)$ because
\[
q_{\mathrm{mix}}'(x)=3x^2-4x+3>0
\]
for all real $x$. Hence on $[\Lm,\infty)$ it suffices to evaluate at $\Lm$:
\[
q_{\mathrm{mix}}(\Lm)=\frac{2m\sqrt{4m-5}+2m+\sqrt{4m-5}-97}{4}.
\]
This is positive for every $m\ge 18$, since
\[
2m\sqrt{4m-5}+2m+\sqrt{4m-5}-97>16m+2m+8-97>0.
\]

Next, at $x=\Lm$ we obtain
\[
R_{\mathrm{mix}}(\Lm)=
\frac{8m^3+34m^2\sqrt{4m-5}-154m^2-495m\sqrt{4m-5}+51m+996\sqrt{4m-5}+324}{8}.
\]
Its derivative with respect to $m$ equals
\[
\frac{d}{dm}R_{\mathrm{mix}}(\Lm)=
\frac{(24m^2-308m+51)\sqrt{4m-5}+340m^2-3310m+4467}{8\sqrt{4m-5}}.
\]
For $m\ge 18$, both numerators
\[
24m^2-308m+51
\qquad\text{and}\qquad
340m^2-3310m+4467
\]
are positive, so $\frac{d}{dm}R_{\mathrm{mix}}(\Lm)>0$ on $[18,\infty)$. Since
\[
R_{\mathrm{mix}}(\Lm)\big|_{m=18}=-\frac{2169}{4}+\frac{1557}{4}\sqrt{67}>0,
\]
it follows that $R_{\mathrm{mix}}(\Lm)>0$ for all even $m\ge 18$.

Finally,
\[
R_{\mathrm{mix}}'(x)=\Bigl(\frac{39m}{2}-168\Bigr)x^2+(2m^2-56m+152)x+m^2-\frac{73m}{2}+145.
\]
Its derivative with respect to $x$ is
\[
R_{\mathrm{mix}}''(x)=(39m-336)x+2m^2-56m+152.
\]
Since $39m-336>0$ for $m\ge 18$, the function $R_{\mathrm{mix}}''(x)$ is increasing in $x$ on $[\Lm,\infty)$. Thus it suffices to evaluate at $x=\Lm$:
\[
R_{\mathrm{mix}}''(\Lm)=\frac{4m^2+(39m-336)\sqrt{4m-5}-73m-32}{2}.
\]
For $m\ge 18$, both terms
\[
39m-336
\qquad\text{and}\qquad
4m^2-73m-32
\]
are positive, so $R_{\mathrm{mix}}''(\Lm)>0$. Hence $R_{\mathrm{mix}}''(x)>0$ for all $x\ge \Lm$, and therefore $R_{\mathrm{mix}}'$ is increasing on $[\Lm,\infty)$.

It remains to check positivity at $\Lm$:
\[
R_{\mathrm{mix}}'(\Lm)=\frac{4m^2\sqrt{4m-5}+86m^2-73m\sqrt{4m-5}-1008m-32\sqrt{4m-5}+1556}{4}.
\]
Its derivative with respect to $m$ is
\[
\frac{d}{dm}R_{\mathrm{mix}}'(\Lm)=\frac{(172m-1008)\sqrt{4m-5}+40m^2-478m+301}{4\sqrt{4m-5}}.
\]
For $m\ge 18$, both numerators
\[
172m-1008
\qquad\text{and}\qquad
40m^2-478m+301
\]
are positive, so $\frac{d}{dm}R_{\mathrm{mix}}'(\Lm)>0$ on $[18,\infty)$. Since
\[
R_{\mathrm{mix}}'(\Lm)\big|_{m=18}=2819-\frac{25}{2}\sqrt{67}>0,
\]
it follows that $R_{\mathrm{mix}}'(\Lm)>0$ for all even $m\ge 18$. Therefore $R_{\mathrm{mix}}'(x)>0$ for all $x\ge \Lm$, so
\[
R_{\mathrm{mix}}(x)\ge R_{\mathrm{mix}}(\Lm)>0
\qquad\text{for all }x\ge \Lm.
\]

By Lemma~\ref{lem:root-template}, the mixed family also lies strictly below $\rhoext(m)$.

\subsection[The e(W)=2, e(A+,W)=2 families]{The $e(W)=2$, $e(\Ap,W)=2$ families}\label{app:ew2-t2}

We record the sign checks used in Propositions~\ref{prop:ew2-t2-p3} and~\ref{prop:ew2-t2-2k2}. Throughout this subsection, set
\[
\Delta=\sqrt{4m-5}.
\]

\subsubsection{The $P_3$ families}\label{app:ew2-t2-p3}

For the family $\mathcal P_m^{(0)}$, the quotient polynomial is
\begin{align*}
 f_0(x)&=x^6-2x^5+(4-m)x^4+(2m-27)x^3 \\
&\quad +(6m-68)x^2+(78-5m)x+84-6m.
\end{align*}
We show that $f_0(x)>0$ for all $x\ge \Lm$. Differentiating three times gives
\[
f_0'''(x)=120x^3-120x^2+(96-24m)x+12m-162.
\]
For $x\ge \Lm$ one has $m\le x^2-x+3/2$, and since $12-24x<0$ for $x\ge \Lm$, we get
\[
f_0'''(x)\ge 12(8x^3-7x^2+4x-12)>0.
\]
Thus $f_0''$ is increasing on $[\Lm,\infty)$. At $x=\Lm$,
\[
f_0''(\Lm)=18m^2+(10m-77)\Delta-12m-\frac{405}{2}>0
\]
for $m\ge 18$. Hence $f_0'$ is increasing on $[\Lm,\infty)$. Also
\[
f_0'(\Lm)=5m^2-75m+\frac{319}{4}+\Delta\left(m^2+8m-\frac{425}{4}\right)>0
\]
for $m\ge 18$. Thus $f_0$ is increasing on $[\Lm,\infty)$.
Finally,
\[
f_0(\Lm)=9m^2-124m+\frac{941}{4}+\Delta\left(\frac{m^2}{2}-11m+\frac{79}{8}\right).
\]
The derivative of this expression as a function of $m$ is
\[
\frac{20m^2+72m\Delta-284m-495\Delta+298}{4\Delta},
\]
which is positive for $m\ge 18$. Since
\[
f_0(\Lm)\big|_{m=18}=\frac{3677}{4}-\frac{209}{8}\sqrt{67}>0,
\]
we have $f_0(\Lm)>0$ for all $m\ge 18$. Therefore $f_0(x)>0$ for all $x\ge \Lm$.

For the family $\mathcal P_m^{(1)}$, the quotient polynomial is
\begin{align*}
 f_1(x)&=x^7-2x^6+(4-m)x^5+(2m-27)x^4+(7m-82)x^3 \\
&\quad +(114-7m)x^2+(148-10m)x+m-16.
\end{align*}
Here
\[
f_1'''(x)=210x^4-240x^3+240x^2-648x-492+m(-60x^2+48x+42).
\]
Using again $m\le x^2-x+3/2$ for $x\ge \Lm$, and the fact that $-60x^2+48x+42<0$ for $x\ge 4$, we get
\[
f_1'''(x)\ge 3(50x^4-44x^3+48x^2-206x-143)>0
\]
for $x\ge \Lm$. Thus $f_1''$ is increasing on $[\Lm,\infty)$. Moreover,
\[
f_1''(\Lm)=39m^2-336m+\frac{1123}{4}+\Delta\left(11m^2+18m-\frac{1577}{4}\right)>0
\]
for $m\ge 18$, and hence $f_1'$ is increasing. Also
\[
f_1'(\Lm)=2m^3+28m^2-445m+\frac{2731}{4}+\Delta\left(\frac72m^2-45m+\frac{93}{8}\right)>0.
\]
So $f_1$ is increasing on $[\Lm,\infty)$. Finally,
\[
f_1(\Lm)=m^3-\frac{77}{4}m^2-\frac m2+\frac{1559}{16}
+\Delta\left(\frac{21}{4}m^2-\frac{311}{4}m+\frac{2817}{16}\right).
\]
The coefficient of $\Delta$ is positive for $m\ge 18$, and $\Delta>8$, so
\[
f_1(\Lm)>
\frac{16m^3+364m^2-9960m+24095}{16}>0
\]
for $m\ge 18$. Therefore $f_1(x)>0$ for all $x\ge \Lm$.

\subsubsection{The $2K_2$ families}\label{app:ew2-t2-2k2}

Let
\[
p_m(x)=x^4-mx^2-(m-2)x+\frac m2-1.
\]
For each of the four $2K_2$ families, write
\[
f_*(x)=q_*(x)p_m(x)+R_*(x),
\]
where $*\in\{ss,sd,ds,dd\}$. Direct division gives
\[
q_{ss}(x)=x^3-2x^2+4x+m-28,
\]
\[
R_{ss}(x)=\left(\frac{19m}{2}-94\right)x^3+(m^2-33m+154)x^2+(m^2-45m+262)x-\frac{m^2}{2}+23m-164;
\]
\[
q_{sd}(x)=x^3-2x^2+5x+m-30,
\]
\[
R_{sd}(x)=\left(\frac{19m}{2}-87\right)x^3+(m^2-32m+120)x^2+\left(m^2-\frac{85m}{2}+193\right)x-\frac{m^2}{2}+20m-102;
\]
\[
q_{ds}(x)=x^3-2x^2+5x+m-31,
\]
\[
R_{ds}(x)=\left(\frac{19m}{2}-84\right)x^3+(m^2-32m+110)x^2+
\left(m^2-\frac{93m}{2}+243\right)x-\frac{m^2}{2}+\frac{37m}{2}-65;
\]
\[
q_{dd}(x)=x^3-2x^2+6x+m-33,
\]
\[
R_{dd}(x)=\left(\frac{19m}{2}-76\right)x^3+(m^2-31m+73)x^2+(m^2-45m+186)x-\frac{m^2}{2}+\frac{37m}{2}-51.
\]

Each $q_*$ has derivative $3x^2-4x+c$ with $c\in\{4,5,5,6\}$, so $q_*$ is increasing on $\mathbb R$. At $x=\Lm$,
\[
q_{ss}(\Lm)=\frac m2\Delta+\frac m2+\frac34\Delta-\frac{103}{4},
\]
\[
q_{sd}(\Lm)=\frac m2\Delta+\frac m2+\frac54\Delta-\frac{109}{4},
\]
\[
q_{ds}(\Lm)=\frac m2\Delta+\frac m2+\frac54\Delta-\frac{113}{4},
\]
\[
q_{dd}(\Lm)=\frac m2\Delta+\frac m2+\frac74\Delta-\frac{119}{4}.
\]
Since $m\ge 18$ implies $\Delta>8$, all four are positive. Hence $q_*(x)>0$ for all $x\ge \Lm$.

It remains to check the remainders. Their second derivatives are
\[
R''_{ss}(x)=2m^2+(57m-564)x-66m+308,
\]
\[
R''_{sd}(x)=2m^2+(57m-522)x-64m+240,
\]
\[
R''_{ds}(x)=2m^2+(57m-504)x-64m+220,
\]
\[
R''_{dd}(x)=2m^2+(57m-456)x-62m+146.
\]
For $x\ge \Lm>4$ and $m\ge 18$, each of these is positive, so every $R_*'$ is increasing on $[\Lm,\infty)$.

Let $\Phi_*(m)=R_*'(\Lm)$. Explicitly,
\begin{align*}
\Phi_{ss}(m)&=m^2\Delta+\frac{61}{2}m^2-\frac{75}{4}m\Delta-\frac{777}{2}m+13\Delta+698,\\
\Phi_{sd}(m)&=m^2\Delta+\frac{61}{2}m^2-\frac{71}{4}m\Delta-364m-\frac{21}{2}\Delta+574,\\
\Phi_{ds}(m)&=m^2\Delta+\frac{61}{2}m^2-\frac{71}{4}m\Delta-359m-16\Delta+605,\\
\Phi_{dd}(m)&=m^2\Delta+\frac{61}{2}m^2-\frac{67}{4}m\Delta-\frac{665}{2}m-41\Delta+487.
\end{align*}
Their derivatives are
\[
\Phi'_{ss}(m)=61m-\frac{777}{2}+\frac{40m^2-490m+479}{4\Delta},
\]
\[
\Phi'_{sd}(m)=61m-364+\frac{40m^2-466m+271}{4\Delta},
\]
\[
\Phi'_{ds}(m)=61m-359+\frac{40m^2-466m+227}{4\Delta},
\]
\[
\Phi'_{dd}(m)=61m-\frac{665}{2}+\frac{40m^2-442m+7}{4\Delta}.
\]
For $m\ge 18$, these are all positive. Since
\[
\Phi_{ss}(18)=3587-\frac{\sqrt{67}}2>0,
\]
\[
\Phi_{sd}(18)=3904-6\sqrt{67}>0,
\]
\[
\Phi_{ds}(18)=4025-\frac{23}{2}\sqrt{67}>0,
\]
\[
\Phi_{dd}(18)=4384-\frac{37}{2}\sqrt{67}>0,
\]
we have $R_*'(\Lm)>0$ for every $m\ge 18$. Since each $R_*'$ is increasing in $x$, it follows that $R_*'(x)>0$ on $[\Lm,\infty)$.

Finally, let $\Psi_*(m)=R_*(\Lm)$. At $m=18$,
\[
\Psi_{ss}(18)=-\frac{207}{4}+\frac{2015}{4}\sqrt{67}>0,
\]
\[
\Psi_{sd}(18)=-151+545\sqrt{67}>0,
\]
\[
\Psi_{ds}(18)=-\frac{985}{4}+\frac{2221}{4}\sqrt{67}>0,
\]
\[
\Psi_{dd}(18)=-\frac{1473}{4}+\frac{2403}{4}\sqrt{67}>0.
\]
Moreover,
\begin{align*}
\Psi'_{ss}(m)&=
\frac{(24m^2-316m+239)\Delta+460m^2-4702m+7239}{8\Delta},\\
\Psi'_{sd}(m)&=
\frac{(24m^2-300m+29)\Delta+460m^2-4450m+6177}{8\Delta},\\
\Psi'_{ds}(m)&=
\frac{(24m^2-300m-43)\Delta+460m^2-4474m+6505}{8\Delta},\\
\Psi'_{dd}(m)&=
\frac{(24m^2-284m-245)\Delta+460m^2-4222m+5511}{8\Delta}.
\end{align*}
Using $\Delta>8$, the numerators are bounded below respectively by
\[
652m^2-7230m+9151,
\]
\[
652m^2-6850m+6409,
\]
\[
652m^2-6874m+6161,
\]
\[
652m^2-6494m+3551,
\]
which are positive for $m\ge 18$. Thus $R_*(\Lm)>0$ for all $m\ge 18$. Since each $R_*$ is increasing on $[\Lm,\infty)$, we get $R_*(x)>0$ for all $x\ge \Lm$. By Lemma~\ref{lem:root-template}, the Perron root of each $2K_2$ family is less than $\rhoext(m)$.

\subsection[The e(W)=2, e(A+,W)=1 families]{The $e(W)=2$, $e(\Ap,W)=1$ families}\label{app:ew2-t1}

We record the finite shifted-polynomial checks used in Proposition~\ref{prop:ew2-t1}. For a family polynomial $f_F$, set
\[
\theta_{18}=\frac92,
\qquad
\theta_{20}=\frac{24}{5}.
\]
For each relevant pair $(F,m)$, the polynomial $f_F$ is obtained from the natural equitable quotient described in the proof of Proposition~\ref{prop:ew2-t1}. We then expand
\[
D_{F,m} f_F(t+\theta_m)=\sum_i c_i t^i,
\qquad t\ge 0,
\]
where $D_{F,m}$ is the positive integer in the table below. The final column in Table~\ref{tab:ew2-t1-shifts} records $\min_i c_i$. Since every listed minimum is positive, each shifted polynomial is positive for all $t\ge 0$.

\begin{table}[htbp]
\centering
\small
\caption{Finite shifted-coefficient checks for the $e(W)=2$, $e(\Ap,W)=1$ families.}
\label{tab:ew2-t1-shifts}
\begin{tabular}{@{}c c c c@{}}
\toprule
Family $F$ & $m$ & $D_{F,m}$ & $\min_i c_i$ \\
\midrule
$\mathcal P^{M,s}$ & $18$ & $256$ & $256$ \\
$\mathcal P^{M,d}$ & $18$ & $128$ & $128$ \\
$\mathcal P^{E,0}$ & $18$ & $256$ & $256$ \\
$\mathcal P^{E,s}$ & $18$ & $256$ & $256$ \\
$\mathcal P^{E,d}$ & $18$ & $512$ & $512$ \\
$\mathcal K^{[3]}$ & $18$ & $256$ & $256$ \\
$\mathcal K^{mix}$ & $18$ & $1024$ & $1024$ \\
$\mathcal K^{edge}$ & $18$ & $512$ & $512$ \\
\midrule
$\mathcal P^{M,s}$ & $20$ & $390625$ & $390625$ \\
$\mathcal P^{M,d}$ & $20$ & $78125$ & $78125$ \\
$\mathcal P^{E,0}$ & $20$ & $390625$ & $390625$ \\
$\mathcal P^{E,s}$ & $20$ & $390625$ & $390625$ \\
$\mathcal P^{E,d}$ & $20$ & $1953125$ & $1953125$ \\
$\mathcal K^{[3]}$ & $20$ & $390625$ & $390625$ \\
$\mathcal K^{mix}$ & $20$ & $9765625$ & $9765625$ \\
$\mathcal K^{edge}$ & $20$ & $1953125$ & $1953125$ \\
$\mathcal K^{[1,1,1]}$ & $20$ & $48828125$ & $48828125$ \\
\bottomrule
\end{tabular}
\end{table}

For completeness, we also note that the omitted row $\mathcal K^{[1,1,1]}$ at $m=18$ is unnecessary because that family would require three distinct special vertices of $\Az$, whereas the edge count gives only $|\Az|=2$ in that case.

\subsection[The e(W)=2, e(A+,W)=0 families]{The $e(W)=2$, $e(\Ap,W)=0$ families}\label{app:ew2-t0}

This subsection records the finite shifted-polynomial check used in Proposition~\ref{prop:ew2-t0}. For each of the ten families in the $m=18$, $e(W)=2$, $e(\Ap,W)=0$ branch, let $f_F(x)$ be the characteristic polynomial of the natural equitable quotient. Set
\[
x=t+\frac{17}{4}.
\]
A direct computation gives
\[
f_{\mathcal P^{0,s}}\!\left(t+\frac{17}{4}\right)
=
\frac{(4t+17)\bigl(1024t^5+18688t^4+120448t^3+321056t^2+300884t+5765\bigr)}{4096},
\]
\[
f_{\mathcal P^{0,d}}\!\left(t+\frac{17}{4}\right)
=
\frac{(4t+13)(4t+17)\bigl(256t^4+3840t^3+17888t^2+24624t+489\bigr)}{4096},
\]
\[
f_{\mathcal P^{1,[3]}}\!\left(t+\frac{17}{4}\right)
=
\frac{4096t^6+92160t^5+799488t^4+3315456t^3+6509936t^2+4860584t+299989}{4096},
\]
\begin{align*}
f_{\mathcal P^{1,adj}}\!\left(t+\frac{17}{4}\right)
&=
\frac{1}{65536}\bigl(65536t^8+2031616t^7+26509312t^6+188542976t^5 \\
&\quad +788767232t^4+1941161216t^3+2612783616t^2 \\
&\quad +1528077008t+95021557\bigr),
\end{align*}
\begin{align*}
f_{\mathcal P^{1,nonadj}}\!\left(t+\frac{17}{4}\right)
&=
\frac{1}{16384}\bigl(16384t^7+438272t^6+4764672t^5+26918656t^4 \\
&\quad +83352768t^3+134567696t^2+90766460t+5750821\bigr),
\end{align*}
\[
f_{\mathcal P^{1,[1,1,1]}}\!\left(t+\frac{17}{4}\right)
=
\frac{4096t^6+92160t^5+803584t^4+3389184t^3+6932336t^2+5623464t+367141}{4096},
\]
\[
f_{\mathcal K^{[4]}}\!\left(t+\frac{17}{4}\right)
=
\frac{(4t+21)\bigl(256t^4+3072t^3+10912t^2+11552t+1229\bigr)}{1024},
\]
\begin{align*}
f_{\mathcal K^{[3,1]}}\!\left(t+\frac{17}{4}\right)
&=
\frac{1}{16384}\bigl(16384t^7+421888t^6+4412416t^5+23970560t^4 \\
&\quad +71411904t^3+111491920t^2+74720028t+7626337\bigr),
\end{align*}
\[
f_{\mathcal K^{[2,2]_e}}\!\left(t+\frac{17}{4}\right)
=
\frac{(4t+9)(4t+21)\bigl(256t^4+3328t^3+13408t^2+16848t+1889\bigr)}{4096},
\]
\[
f_{\mathcal K^{[2,2]_c}}\!\left(t+\frac{17}{4}\right)
=
\frac{(4t+13)(4t+25)\bigl(256t^4+3328t^3+13408t^2+16848t+1889\bigr)}{4096}.
\]
All displayed factors and coefficients are positive for $t\ge 0$, so every shifted polynomial is positive on $[0,\infty)$.

\subsection[Low-end non-K4 quotient checks]{Low-end non-$K_4$ quotient checks}\label{app:nonK4-low}

This subsection records the finite quotient comparisons used in Lemma~\ref{lem:nonK4-case1}. The only low-end non-$K_4$ branch not already ruled out by the inequalities in the main text is $G[\Ap]\cong K_4-e$. We use the following notation. The family $\mathcal H^0_m$ is the case $e(W)=0$, so
\[
\mathcal H^0_m=K_1\vee\bigl((K_4-e)\cup (m-9)K_1\bigr).
\]
When $e(W)=1$, the unique edge of $W$ has one of five possible attachment patterns: same-$\Az$, distinct-$\Az$, mixed high, mixed low, or double-low. These are denoted
\[
\mathcal H^{s}_m,\quad \mathcal H^{d}_m,\quad
\mathcal H^{h}_m,\quad \mathcal H^{\ell}_m,\quad
\mathcal H^{r,s}_m,
\]
respectively. Here ``mixed high'' means the $\Ap$-adjacent endpoint of the $W$-edge attaches to one of the two degree-three vertices of $K_4-e$, while ``mixed low'' means it attaches to one of the two degree-two vertices. The double-low family is the case in which the two endpoints attach to the unique nonadjacent pair of $K_4-e$.

For each family $F$ and relevant value of $m$, let $f_{F,m}(x)$ be the characteristic polynomial of the natural equitable quotient matrix. The partition separates $u^*$, the automorphism orbits inside $G[\Ap]$, the special vertices of $W$, the special vertices of $\Az$, and the ordinary leaves of $\Az$; this makes each quotient computation reproducible from the stated family description. Table~\ref{tab:nonk4-low-shifts} records a positive shift $\theta_m$ and a positive integer $D_{F,m}$ such that
\[
D_{F,m}f_{F,m}(t+\theta_m)
\]
has all coefficients positive. Hence $f_{F,m}(x)>0$ for all $x\ge \theta_m$.

\begin{table}[htbp]
\centering
\small
\caption{Low-end non-$K_4$ shifted-coefficient checks.}
\label{tab:nonk4-low-shifts}
\begin{tabular}{@{}c c c c c@{}}
\toprule
Family $F$ & $m$ & $\theta_m$ & $D_{F,m}$ & $\min_i c_i$ \\
\midrule
$\mathcal H^0_m$ & $18$ & $9/2$ & $16$ & $16$ \\
$\mathcal H^0_m$ & $20$ & $24/5$ & $625$ & $625$ \\
$\mathcal H^0_m$ & $22$ & $24/5$ & $625$ & $625$ \\
\midrule
$\mathcal H^{s}_m$ & $18$ & $17/4$ & $4096$ & $4096$ \\
$\mathcal H^{d}_m$ & $18$ & $17/4$ & $4096$ & $4096$ \\
$\mathcal H^{h}_m$ & $18$ & $17/4$ & $65536$ & $65536$ \\
$\mathcal H^{\ell}_m$ & $18$ & $17/4$ & $65536$ & $65536$ \\
$\mathcal H^{r,s}_m$ & $18$ & $17/4$ & $1024$ & $1024$ \\
\midrule
$\mathcal H^{s}_m$ & $20$ & $24/5$ & $15625$ & $15625$ \\
$\mathcal H^{d}_m$ & $20$ & $24/5$ & $15625$ & $15625$ \\
$\mathcal H^{h}_m$ & $20$ & $24/5$ & $390625$ & $390625$ \\
$\mathcal H^{\ell}_m$ & $20$ & $24/5$ & $390625$ & $390625$ \\
$\mathcal H^{r,s}_m$ & $20$ & $24/5$ & $3125$ & $3125$ \\
\bottomrule
\end{tabular}
\end{table}

The comparison values are legitimate because
\[
p_{18}\!\left(\frac{17}{4}\right)<0,\qquad
p_{18}\!\left(\frac92\right)<0,\qquad
p_{20}\!\left(\frac{24}{5}\right)<0,\qquad
p_{22}\!\left(\frac{24}{5}\right)<0.
\]
Therefore $\rhoext(m)>\theta_m$ in every row of the table, and the corresponding Perron root is smaller than $\rhoext(m)$.

\subsection{Exact verification protocol for the finite tables}\label{app:verification-script}

For reproducibility, the finite shifted-coefficient checks in this appendix were verified with exact rational arithmetic. The accompanying ancillary file \href{https://github.com/Shreyhaan/Spectral-Graph-Theory-1/blob/main/finite_verification.py}{\texttt{finite\_verification.py}} follows the same protocol for every advertised finite row: construct the relevant polynomial, substitute the rational shift, clear denominators, and check positivity of each coefficient. The essential routine is:
\begin{verbatim}
from sympy import symbols, Rational, Poly, expand
x, t = symbols('x t')

def check_shift(poly, theta, denom=1):
    P = Poly(expand(denom * poly.subs(x, t + theta)), t)
    coeffs = P.all_coeffs()
    assert all(c > 0 for c in coeffs), coeffs
    return min(coeffs), coeffs
\end{verbatim}
The computations use no floating point arithmetic. The script verifies the shifted-coefficient checks; the structural reductions and family classifications are the mathematical content of the main text.

\printcredits

\section*{Funding}
This research did not receive any specific grant from funding agencies in the public, commercial, or not-for-profit sectors.

\section*{Declaration of competing interest}
The author declares no known competing financial interests or personal relationships that could have appeared to influence the work reported in this paper.

\section*{Data availability}
No datasets were generated or analyzed in this paper. The ancillary verification script used for the finite shifted-coefficient checks is submitted separately with the manuscript and is also publicly available at \url{https://github.com/Shreyhaan/Spectral-Graph-Theory-1/blob/main/finite_verification.py}.

\section*{Acknowledgments}
The author acknowledges the work of R. Zheng and G. Zhang, whose Perron-neighborhood decomposition provides an important starting point for the branch analysis used here.


\begin{thebibliography}{99}
\bibitem{BrualdiHoffman1985}
R.~A. Brualdi and A.~J. Hoffman,
On the spectral radius of $(0,1)$-matrices,
Linear Algebra Appl. \textbf{65} (1985), 133--146.
\href{https://doi.org/10.1016/0024-3795(85)90092-8}{doi:10.1016/0024-3795(85)90092-8}.

\bibitem{LiLuPeng2023}
Y.~Li, L.~Lu, and Y.~Peng,
Spectral extremal graphs for the bowtie,
Discrete Math. \textbf{346} (2023), no.~12, Paper No.~113680.
\href{https://doi.org/10.1016/j.disc.2023.113680}{doi:10.1016/j.disc.2023.113680}.

\bibitem{LiZhaoZou2025}
S.~Li, S.~Zhao, and L.~Zou,
Spectral extrema of graphs with fixed size: forbidden a fan graph, a friendship graph, or a theta graph,
J. Graph Theory \textbf{110} (2025), 483--495.
\href{https://doi.org/10.1002/jgt.23287}{doi:10.1002/jgt.23287}.

\bibitem{LiuWang2024}
Y.~X. Liu and L.~Wang,
Spectral radius of graphs of given size with forbidden subgraphs,
Linear Algebra Appl. \textbf{689} (2024), 108--125.
\href{https://doi.org/10.1016/j.laa.2024.02.026}{doi:10.1016/j.laa.2024.02.026}.

\bibitem{Nikiforov2010}
V.~Nikiforov,
The spectral radius of graphs without paths and cycles of specified length,
Linear Algebra Appl. \textbf{432} (2010), 2243--2256.
\href{https://doi.org/10.1016/j.laa.2009.05.023}{doi:10.1016/j.laa.2009.05.023}.

\bibitem{Nosal1970}
E.~Nosal,
Eigenvalues of graphs,
Master's thesis, University of Calgary, 1970.
\href{https://doi.org/10.11575/PRISM/21350}{doi:10.11575/PRISM/21350}.

\bibitem{PirzadaRehman2025}
S.~Pirzada and A.~Rehman,
The maximum spectral radius of graphs of even size forbidding $\{H(3,3),H(4,3)\}$,
Comput. Appl. Math. \textbf{44} (2025), Paper No.~295.
\href{https://doi.org/10.1007/s40314-025-03257-0}{doi:10.1007/s40314-025-03257-0}.

\bibitem{RehmanPirzada2025}
A.~Rehman and S.~Pirzada,
The spectral Tur\'an problem about graphs of given size with forbidden subgraphs,
AKCE Int. J. Graphs Comb. \textbf{22} (2025), no.~1, 91--93.
\href{https://doi.org/10.1080/09728600.2024.2421212}{doi:10.1080/09728600.2024.2421212}.


\bibitem{ZhangWangFish2025}
Y.~Zhang and L.~Wang,
Spectral extremal problem on the fish graph,
Bull. Malays. Math. Sci. Soc. \textbf{48} (2025), Art.~207.
\href{https://doi.org/10.1007/s40840-025-01992-5}{doi:10.1007/s40840-025-01992-5}.

\bibitem{ZhengZhang2025}
R.~Zheng and G.~Zhang,
Spectral extrema of graphs of given even size forbidding $H(4,3)$,
arXiv preprint \href{https://arxiv.org/abs/2509.18594}{arXiv:2509.18594}, 2025.
\href{https://doi.org/10.48550/arXiv.2509.18594}{doi:10.48550/arXiv.2509.18594}.
\end{thebibliography}
\end{document}